\newif\ifdraft
\newcommand{\publishtype}{}         

\drafttrue
\draftfalse
\ifdraft
\else
  \renewcommand{\publishtype}{arxiv}
\fi

\makeatletter
\newcommand{\publishtype@arxiv}{arxiv}
\newcommand{\whenarxiv}[1]{\ifx\publishtype\publishtype@arxiv #1\fi}
\newcommand{\publishtype@check}{check}
\newcommand{\whencheck}[1]{\ifx\publishtype\publishtype@check #1\fi}
\newcommand{\publishtype@standard}{standard}
\newcommand{\whenstandard}[1]{\ifx\publishtype\publishtype@standard #1\fi}
\makeatother

\ifdraft
  \documentclass[dvipdfmx]{article}
\else
  \whenarxiv{\documentclass{article}}
  \whencheck{\documentclass[dvipdfmx,12pt]{article}}
  \whenstandard{\documentclass[dvipdfmx]{article}}
\fi

\makeatletter
\newcommand{\version@filename}{./\jobname.tex.version}
\newcommand{\version}{\IfFileExists{\version@filename}
  {version: \input{\version@filename}}
  {Error: version not found}}
\makeatother
\newif\ifusefancyhdr
\newif\ifusedvipdfmx
\ifdraft
  \usefancyhdrtrue
  \usedvipdfmxtrue
\else
  \whencheck{%
    \usefancyhdrtrue
    \usedvipdfmxtrue
  }
\fi
\ifusefancyhdr
  \usepackage{fancyhdr}
  \usepackage{datetime}
\fi

\usepackage{amsmath}            
\usepackage{amssymb}
\usepackage{amsthm}
\usepackage{mathtools}
\usepackage{enumitem}
\setenumerate{label={\rm(\arabic*)}}
\ifusedvipdfmx
  \usepackage[dvipdfmx]{hyperref}
\else
  \usepackage{hyperref}
\fi
\usepackage[capitalize]{cleveref} 
\usepackage{autonum}
\usepackage{tikz}
\usepackage{pxpgfmark}          
\usetikzlibrary{calc}
\usetikzlibrary{intersections}
\usetikzlibrary{positioning}
\usepackage{tikz-cd}

\usepackage[utf8]{inputenc}

\newtheoremstyle{mydescription}
  {}
  {}
  {}
  {}
  {}
  {}
  { }
  {(\thmnumber{#2})\hspace{0.2zw}}

\numberwithin{equation}{section}

\theoremstyle{definition}
\newtheorem{definition}[equation]{Definition}
\theoremstyle{plain}
\newtheorem{lemma}[equation]{Lemma}
\newtheorem{proposition}[equation]{Proposition}
\newtheorem{theorem}[equation]{Theorem}
\newtheorem{corollary}[equation]{Corollary}

\theoremstyle{remark}
\newtheorem{remark}[equation]{Remark}

\crefname{equation}{}{}
\crefname{enumi}{}{}
\crefname{subsection}{Subsection}{Subsections}

\usepackage{todonotes}
\presetkeys{todonotes}{color=orange!60}{}

\makeatletter
\let\todo@orig=\todo
\renewcommand{\todo}[1]{\todo@orig{#1}\relax}
\makeatother
\ifdraft
\else
  \presetkeys{todonotes}{disable}{}
\fi
\makeatletter
\tikzstyle{inlinenotestyle} = [
    draw=\@todonotes@currentbordercolor,
    fill=\@todonotes@currentbackgroundcolor,
    line width=0.5pt,
    inner sep = 0.8 ex,
    rounded corners=4pt]

\renewcommand{\@todonotes@drawInlineNote}{%
        {\begin{tikzpicture}[remember picture,baseline=(current bounding box.base)]%
            \draw node[inlinenotestyle,font=\@todonotes@sizecommand, anchor=base,baseline]{%
              \if@todonotes@authorgiven%
                {\noindent \@todonotes@sizecommand \@todonotes@author:\,\@todonotes@text}%
              \else%
                {\noindent \@todonotes@sizecommand \@todonotes@text}%
              \fi};%
           \end{tikzpicture}}}%
\newcommand{\mytodo}[1]{\@todo[inline]{#1}}%
\makeatother

\ifdraft
  \usepackage[color]{showkeys}
  \makeatletter
  \SK@def\cref#1{\SK@\SK@@ref{#1}\SK@cref{#1}} 
  \makeatother
\else
  \whencheck{
    \usepackage[color,notref,notcite]{showkeys}
  }
\fi
\definecolor{refkey}{rgb}{0.7, 0.8, 0.5}
\definecolor{labelkey}{rgb}{0, 0.7, 0.5}

\newcommand{\id}{{\rm id}}
\newcommand{\pr}{\mathrm{pr}}
\newcommand{\ev}{\mathrm{ev}}
\newcommand{\res}{\mathrm{res}}
\renewcommand{\deg}[1]{\mathopen|#1\mathclose|}
\newcommand{\K}{\mathbb K}

\newcommand{\Z}{\mathbb Z}

\newcommand{\cochain}[2][*]{C^{#1}(#2)}
\newcommand{\cohom}[2][*]{H^{#1}(#2)}
\newcommand{\homol}[2][*]{H_{#1}(#2)}

\renewcommand{\ker}{\operatorname{Ker}}
\renewcommand{\hom}{\operatorname{Hom}}
\newcommand{\ext}{\operatorname{Ext}}
\newcommand{\tor}{\operatorname{Tor}}
\renewcommand{\lim}{{\varprojlim}}

\newcommand{\spanv}{\operatorname{span}}
\newcommand{\exttohom}[1]{\cohom{#1}}
\newcommand{\extlift}[1]{{#1}^\star} 

\newcommand{\evenpart}[1]{#1^{\rm even}}
\newcommand{\oddpart}[1]{#1^{\rm odd}}

\newcommand{\shriek}[1]{{#1}_!}
\newcommand{\incl}{\mathrm{incl}}
\newcommand{\comp}{\mathrm{comp}}
\newcommand{\diag}{\Delta}
\newcommand{\inclconst}{c}

\newcommand{\qis}{\simeq}
\newcommand{\tpow}[1]{^{\otimes #1}}
\newcommand{\susp}[2][]{s^{#1}#2}

\newcommand{\dual}{\vee}        

\newcommand{\dcop}{\delta^\dual}
\newcommand{\dcopnew}{\delta^\dual_{\mathrm{ns}}}
\DeclareMathOperator{\map}{Map}

\newcommand{\spheresp}[3][]{S_{#1}^{#2}#3}
\newcommand{\disksp}[2]{D^{#1}#2}
\newcommand{\mapsp}[3][]{{#3}^{#2}_{#1}}
\newcommand{\htpyset}[2]{[#1, #2]}

\newcommand{\orirev}{\tau}

\DeclareMathOperator{\chara}{ch}

\title{
  New construction of the brane coproduct
  and vanishing of cup products on sphere spaces
}
\author{Shun Wakatsuki}
\date{}

\begin{document}
\maketitle
\ifusefancyhdr
  \thispagestyle{fancy}
\fi
\begin{abstract}
  Using the loop coproduct,
  Menichi proved that the cup product with the orientation class vanishes
  for a closed connected oriented manifold with non-trivial Euler characteristic.
  We generalize this to the sphere spaces,
  i.e.\ the mapping spaces from spheres,
  using two generalizations of the loop coproduct to sphere spaces.
  One is constructed in this paper
  and the other in a previous paper of the author.
\end{abstract}

\ifdraft
  \listoftodos
\fi

\section{Introduction}
In this article,
we give a new construction of the brane coproduct,
which we call the \textit{non-symmetric} brane coproduct.
Comparing this coproduct with another coproduct constructed in \cite{wakatsuki18:toappear},
we prove the vanishing of some cup products on the cohomology of mapping spaces from spheres.

Chas and Sullivan \cite{chas-sullivan} introduced the loop product
on the homology \(\homol{LM}\) of the free loop space \(LM=\map(S^1, M)\)
of a manifold \(M\) of dimension \(m\).
Cohen and Godin \cite{cohen-godin} extended this product to other string operations,
including the loop coproduct,
whose dual has the form
\begin{equation}
  \dcop\colon \cohom{LM} \to \cohom[*+m]{LM\times LM}.
\end{equation}
Although the loop coproduct is almost trivial by \cite{tamanoi},
Menichi \cite{menichi13} used the loop coproduct to obtain the following vanishing result.
\begin{theorem}[{\cite[Theorem 1]{menichi13}}]
  \label{theorem:menichiVanishing}
  Let \(M\) be a connected, closed oriented manifold of dimension \(m\),
  \(\omega \in \cohom[m]{M}\) its orientation class, and
  \(\chi(M)\) its Euler characteristic.
  Then, for any \(\alpha \in \cohom[>0]{LM}\), we have
  \begin{equation}
    \chi(M)\ev_0^*\omega\cdot\alpha = 0 \in \cohom[\deg{\alpha}+m]{LM},
  \end{equation}
  where
  \begin{math}
    \ev_0\colon LM \to M
  \end{math}
  is the evaluation map at the base point \(0\in S^1\).
\end{theorem}

Moreover, Félix and Thomas \cite{felix-thomas09} generalized
the loop coproduct to Gorenstein spaces.
A Gorenstein space is a generalization of
a Poincaré duality space (i.e.\ a space satisfying Poincaré duality)
in an algebraic way.
See \cref{definition:Gorenstein} for the definition.

Using the algebraic method due to Félix and Thomas,
the author \cite{wakatsuki18:toappear} constructed
a generalization of the loop coproduct,
called the \textit{brane coproduct}.
Here we explain it along with a little generalization.
Let \(\K\) be a field,
\(k\) a positive integer and
\(M\) a \(k\)-connected space with \(\cohom{M} = \cohom{M;\K}\) of finite type.
Denote by \(\spheresp{k}{M} = \map(S^k, M)\)
the mapping space from the \(k\)-dimensional sphere to \(M\).
We fix an arbitrary element
\newcommand{\myshriek}{\gamma}
\begin{math}
  \myshriek\in
  \ext^{l}_{\cochain{\spheresp{k-1}{M}}}(\cochain{M}, \cochain{\spheresp{k-1}{M}}),
\end{math}
where \(\cochain{M}\) is the singular cochain algebra on \(M\).
Then we can construct (the dual of) the brane coproduct
\begin{equation}
  \dcop_\myshriek\colon \cohom{\spheresp{k}{M}\times\spheresp{k}{M}}
  \to \cohom[*+l]{\spheresp{k}{M}}.
\end{equation}
Note that \(\myshriek\) will be specified under some assumption on \(M\),
and that we can choose \(l\) and \(\myshriek\) depending on the purpose.
See \cref{section:reviewPreviousBraneCoproduct} for details.

Next we explain the \textit{non-symmetric} brane coproduct,
which will be defined in this article.
Assume \(M\) is a Poincaré duality space (i.e.\ a space satisfying Poincaré duality)
over \(\K\) of dimension \(m\).
Then we can define the non-symmetric brane coproduct
\begin{equation}
  \dcopnew\colon \cohom{\spheresp{k}{M}\times\spheresp{k}{M}}
  \to \cohom[*+m]{\spheresp{k}{M}}.
\end{equation}
Note that the non-symmetric brane coproduct can be defined
for any 1-connected Poincaré duality space,
without the assumption of \(k\)-connectivity.
See \cref{section:newCoprodConstruction} for details.

The non-symmetric brane coproduct \(\dcopnew\)
seems to be \textit{non-commutative},
from the explicit formula in \cref{theorem:computeNewCoprod}.
On the other hand,
the brane coproduct \(\dcop_\myshriek\) is \textit{commutative} in the sense of
\cref{proposition:commutativityOfPreviousCoproduct}.
In spite of such difference,
these coproducts coincide with each other under some assumptions.
This coincidence gives some non-trivial relations on \(\cohom{\spheresp{k}{M}}\),
which is the main theorem of this article:

\begin{theorem}
  \label{theorem:vanishingMainTheorem}
  Let \(k\) be a positive integer,
  \(M\) a \(k\)-connected Poincaré duality space over \(\K\) of dimension \(m\), and
  \(\omega\in\cohom[m]{M}\) its orientation class.
  Assume
  \begin{enumerate}
    \item \label{item:assump_loop} \(k=1\) or
    \item \label{item:assump_brane}
      \(k\geq 1\) is odd,
      the characteristic of \(\K\) is zero, and
      \(\dim_\K\left(\bigoplus_n\pi_n(M)\otimes\K\right) < \infty\).
      \todo{\(\Omega^{k-1}M\) Gorenstein にも触れる？}
  \end{enumerate}

  Then, for any \(\alpha\in\cohom[>0]{\spheresp{k}{M}}\), we have
  \begin{equation}
    \chi(M)\ev_0^*\omega\cdot\alpha = 0 \in \cohom[\deg{\alpha}+m]{\spheresp{k}{M}}.
  \end{equation}
\end{theorem}

\begin{remark}
  This theorem generalizes \cref{theorem:menichiVanishing} due to Menichi,
  since we do not assume
  that \(M\) is a manifold and \(k = 1\).
  See \cref{remark:assumpIsNatural} for the reason why we need the assumption \(k\) is odd.
\end{remark}

We prove the above theorem
using the following general result.

\begin{theorem}
  \label{theorem:vanishingGeneral}
  Let \(M\) be a \(k\)-connected Poincaré duality space over \(\K\) of dimension \(m\),
  \(\omega \in \cohom[m]{M}\) its orientation class.
  We fix an arbitrary element
  \begin{equation}
    \myshriek \in
    \ext^{m}_{\cochain{\spheresp{k-1}{M}}}(\cochain{M}, \cochain{\spheresp{k-1}{M}}).
  \end{equation}
  Define \(\lambda_\myshriek\in \K\) by the equation
  \begin{math}
    \inclconst^*\circ(\exttohom{\myshriek})(1) = \lambda_\myshriek \omega
    \in \cohom[m]{M},
  \end{math}
  where \(\inclconst\colon M\to\spheresp{k-1}{M}\) is the embedding as constant maps.
  See \cref{section:extBasic} for the definition of the map
  \begin{math}
    \exttohom{\myshriek}\colon
    \cohom{M}\to\cohom{\spheresp{k-1}{M}}.
  \end{math}
  Then, for any \(\alpha\in\cohom[>0]{\spheresp{k}{M}}\), we have
  \begin{equation}
    \lambda_\myshriek\ev_0^*\omega\cdot\alpha = 0 \in \cohom[\deg{\alpha}+m]{\spheresp{k}{M}}.
  \end{equation}
\end{theorem}

We conjecture that,
for any \(M\) and \(k\) as in \cref{theorem:vanishingGeneral},
there is an element \(\myshriek\) satisfying \(\lambda_\myshriek = \chi(M)\).
The assumptions \cref{item:assump_loop} and \cref{item:assump_brane}
give sufficient conditions for the existence of such \(\myshriek\).

Throughout this article, \(\K\) denotes a field.
The characteristic \(\chara\K\) of the field \(\K\) is zero in
\cref{subsection:comparison_higher} and
\cref{section:modelOfShriek}.
In other (sub)sections,
\(\chara\K\) can be zero or any prime.
For a vector space \(V\) over \(\K\),
we denote the dual of \(V\) by \(V^\dual\).
For spaces \(X\) and \(Y\),
we denote the mapping space from \(X\) to \(Y\) by \(\mapsp{X}{Y}\).
For \(x\in X\), let \(\ev_x\colon \mapsp{X}{Y} \to Y\) be the evaluation map at \(x\).
Denote by \(\htpyset{X}{Y}\) the homotopy set of maps from \(X\) to \(Y\).
Base points does not matter since we consider it
only when \(X\) is 0-connected and \(Y\) is 1-connected.
\todo{all spaces are finite type?}


This article is organized as follows.
\cref{section:extBasic} contains basic definitions and properties of \(\ext\),
which we use in definitions of the brane coproducts.
In \cref{section:reviewPreviousBraneCoproduct},
we review the previous construction of the brane coproduct.
We define the non-symmetric brane coproduct in \cref{section:newCoprodConstruction},
and, under some assumptions, explicitly compute it in \cref{section:computeNewCoprod}.
In \cref{section:comparison},
we compare two brane coproducts
and prove \cref{theorem:vanishingMainTheorem} and \cref{theorem:vanishingGeneral},
using explicit construction of shriek maps given in \cref{section:modelOfShriek}.

\tableofcontents

\section{Definition and properties of \(\ext\)}
\label{section:extBasic}
Let \(A\) be a differential graded algebra (dga),
and \(M\) and \(N\) \(A\)-modules over a field \(\K\) of any characteristic.
Then the extension module is defined as
\begin{math}
  \ext_A(M,N) = \cohom{\hom_A(P,A)},
\end{math}
where \(P\) is a semifree resolution of \(M\) over \(A\).
See \cite[Section 6]{felix-halperin-thomas01} for details of semifree resolutions.
For an element \(\alpha \in \ext_A(M,N)\),
we define
\begin{math}
  \exttohom{\alpha}\colon
  \cohom{M} \xleftarrow{\cong} \cohom{P}
  \xrightarrow{\cohom{\alpha}} \cohom{N}.
\end{math}
This defines a linear map
\begin{math}
  \ext_A(M,N) \to \hom_{\cohom{A}}(\cohom{M},\cohom{N});\ \alpha \mapsto \exttohom{\alpha}.
\end{math}

\newcommand{\pb}{D}
Consider a pullback diagram
\begin{equation}
  \begin{tikzcd}
    \pb \ar[r] \ar[d] & E \ar[d,"p"] \\
    A \ar[r] & B
  \end{tikzcd}
\end{equation}
such that \(p\colon E\to B\) is a fibration
and \(B\) is 1-connected.
Let us recall the linear map
\begin{equation}
  \extlift{p}\colon
  \ext_{\cochain{B}}^l(\cochain{A}, \cochain{B})
  \to \ext_{\cochain{E}}^l(\cochain{\pb},\cochain{E})
\end{equation}
introduced in \cite[Remark after Theorem 2]{felix-thomas09}.
Let \(P\) be a semifree resolution of \(\cochain{A}\) over \(\cochain{B}\).
Then we have a linear map
\begin{equation}
  \label{equation:chainLevelLift}
  \hom_{\cochain{B}}(P,\cochain{B})
  \to \hom_{\cochain{E}}(\cochain{E}\otimes_{\cochain{B}}P, \cochain{E}\otimes_{\cochain{B}}\cochain{B})
\end{equation}
by sending
\begin{math}
  \varphi \in \hom_{\cochain{B}}(P,\cochain{B})
\end{math}
to
\begin{math}
  \id_{\cochain{E}}\otimes\varphi.
\end{math}
Here, \(\cochain{E}\otimes_{\cochain{B}}P\) is
a semifree \(\cochain{E}\) module by \cite[Lemma 6.2]{felix-halperin-thomas01}.
Moreover, the Eilenberg-Moore map
\begin{math}
  \cochain{E}\otimes_{\cochain{B}}P \to \cochain{D}
\end{math}
is a quasi-isomorphism by the Eilenberg-Moore theorem
\cite[Theorem 3.2]{smith67}.
Hence \(\cochain{E}\otimes_{\cochain{B}}P\) is
a semifree resolution of \(\cochain{\pb}\) over \(\cochain{E}\),
and the linear map \cref{equation:chainLevelLift}
induces the required map \(\extlift{p}\).

The above constructions satisfy naturality in the following sense,
which can be proved directly from the definitions.
\begin{proposition}
  \label{proposition:naturalityOfShriek}
  Consider a diagram
  \begin{equation}
    \label{pullback hom 4}
    \begin{tikzcd}[row sep=2.5em]
      & E' \ar[dd,"p'" near end] \ar[dl,swap,"\varphi"]&&
      X' \ar[ll,""] \ar[dd,""] \ar[dl,"\psi"]\\
      E \ar[dd,swap,"p"] &&
      X \ar[ll,crossing over,"" near start,swap] \\
      & B' \ar[dl,swap,"a",swap] && A' \ar[ll,"" near end] \ar[dl,"b"]\\
      B &&
      A \ar[ll,""] \ar[uu,<-,crossing over,swap,"" near end]
    \end{tikzcd}
  \end{equation}
  and elements \(\alpha \in \ext^m_{\cochain{B}}(\cochain{A}, \cochain{B})\)
  and \(\alpha' \in \ext^m_{\cochain{B'}}(\cochain{A'}, \cochain{B'})\).
  Here \(p\) and \(p'\) are fibrations
  and the front and back squares are pullback diagrams.
  Assume that the elements \(\alpha\) and \(\alpha'\)
  are mapped to the same element in \(\ext^m_{\cochain{B}}(\cochain{A}, \cochain{B'})\)
  by the morphisms induced by \(a\) and \(b\),
  and that the Eilenberg-Moore maps of two pullback diagrams are isomorphisms.
  Then the following diagram commutes.
  \begin{equation}
    \begin{tikzcd}[column sep=huge]
      \cohom{X} \ar[r,"\exttohom{\extlift{p}\alpha}"] \ar[d,"\varphi^*"] & \cohom[*+m]{E} \ar[d,"\psi^*"]\\
      \cohom{X'} \ar[r,"\exttohom{\extlift{p'}\alpha'}"] & \cohom[*+m]{E'}
    \end{tikzcd}
  \end{equation}
\end{proposition}

\section{Review of the previous construction of the brane coproduct}
\label{section:reviewPreviousBraneCoproduct}
In this section,
we review the previous construction of the brane coprdouct from \cite{wakatsuki18:toappear}.
Here we explain it in a generalized way,
which is necessary for the comparison in \cref{section:comparison}.

First we give a general construction.
Let \(\K\) be a field of any characteristic,
\(k\) a positive integer,
\(S\) and \(T\) \(k\)-dimensional manifolds, and
\(M\) a \(k\)-connected space.
We fix an arbitrary element
\begin{equation}
  \myshriek\in
  \ext^n_{\cochain{\spheresp{k-1}{M}}}(\cochain{M}, \cochain{\spheresp{k-1}{M}}).
\end{equation}

To define the brane coproduct,
consider the diagram
\begin{equation}
  \label{equation:STCopDiagram}
  \begin{tikzcd}
    M^{S\#T}\ar[d,"\res"] & M^S\times_MM^T \ar[l,"\comp"']\ar[d]\ar[r,"\incl"] & M^S\times M^T\\
    \spheresp{k-1}{M} & M, \ar[l,"\inclconst"]
  \end{tikzcd}
\end{equation}
where
the square is a pullback diagram,
the map \(\res\) is the restriction map to the embedded sphere \(S^{k-1}\)
which comes from the connected sum \(S\#T\), and
the map \(\inclconst\) is an embedding as constant maps.

Then the dual
\begin{equation}
  \dcop_\myshriek\colon \cohom[*]{M^S\times M^T} \to \cohom[*+n]{M^{S\#T}}
\end{equation}
of the brane coproduct with respect to \(\myshriek\) is defined as the composition
\begin{equation}
  \shriek\comp\circ\incl^*\colon\cohom{\mapsp{S}{M}\times\mapsp{T}{M}}
  \xrightarrow{\incl^*}\cohom{\mapsp{S}{M}\times_M\mapsp{T}{M}}
  \xrightarrow{\shriek\comp}\cohom[*+n]{\mapsp{S\#T}{M}}.
\end{equation}
Here the shriek map \(\shriek\comp\) is defined by
\begin{math}
  \shriek\comp = \exttohom{\extlift{\res}(\myshriek)}.
\end{math}

Next we specify the element \(\myshriek\) under some assumptions,
which was considered in \cite{wakatsuki18:toappear}.
Here we use the notion of a Gorenstein space.

\begin{definition}[{\cite{felix-halperin-thomas88}}]
  \label{definition:Gorenstein}
  Let \(m\in\Z\) be an integer.
  A path-connected topological space \(M\) is called
  a \textit{(\(\K\)-)Gorenstein space} of dimension \(m\)
  if
  \begin{equation}
    \dim \ext_{\cochain{M}}^l(\K, \cochain{M}) =
    \begin{cases}
      1, & \mbox{if \(l = m\)} \\
      0, & \mbox{otherwise.}
    \end{cases}
  \end{equation}
\end{definition}

For example,
a Poincaré duality space over \(\K\) is a \(\K\)-Gorenstein space,
and its dimension as a Gorenstein space coincides
with the one as a Poincaré duality space.
Moreover,
the following proposition gives an important example of a Gorenstein space.

\begin{proposition}
  [{\cite[Proposition 3.4]{felix-halperin-thomas88}}]
  \label{proposition:FinDimImplyGorenstein}
  A 1-connected topological space \(M\) is a \(\K\)-Gorenstein space
  if \(\K\) is a field of characteristic zero
  and \(\pi_*(M)\otimes\K\) is finite dimensional.
\end{proposition}

Now we can specify the element \(\myshriek\)
by the following theorem.

\begin{theorem}[{\cite[Corollary 3.2]{wakatsuki18:toappear}}]
  \label{theorem:extSphereSpace}
  Assume \(\K\) is a field of characteristic zero.
  Let \(M\) be a \((k-1)\)-connected (and 1-connected) space of finite type
  such that \(\Omega^{k-1}M\) is a Gorenstein space of dimension \(\bar{m}\).
  Then we have an isomorphism
  \begin{equation}
    \ext^l_{\cochain{\spheresp{k-1}{M}}}(\cochain{M}, \cochain{\spheresp{k-1}{M}}) \cong \cohom[l-\bar{m}]{M}
  \end{equation}
  for any \(l\in\Z\).
\end{theorem}
When \(l=\bar{m}\), we have the generator
\begin{equation}
  \label{equation:shriekInExt}
  \shriek\inclconst\in
  \ext^{\bar{m}}_{\cochain{\spheresp{k-1}{M}}}(\cochain{M}, \cochain{\spheresp{k-1}{M}})
  \cong \cohom[0]{M} \cong \K
\end{equation}
up to non-zero scalar multiplication.
The brane coproduct \(\dcop_{\shriek\inclconst}\)
for the case \(\myshriek=\shriek\inclconst\)
is the brane coproduct constructed in \cite{wakatsuki18:toappear}.

\section{New construction of the brane coproduct}
\label{section:newCoprodConstruction}
In this section,
we give a new construction of the brane coproduct,
which we call the \textit{non-symmetric} brane coproduct.
This is different from the previous one
and we will compare them in \cref{section:comparison}.

Let \(\K\) be a field of any characteristic,
\(k\) a positive integer,
\(T\) a \(k\)-dimensional manifold with a base point \(t_0\), and
\(M\) a \(1\)-connected Poincaré duality space of dimension \(m\).
We fix base points \(d_0\in D^k\) and \(s_0\in S^k\)
such that \(d_0\) is mapped to \(s_0\) by the quotient map \(D^k \twoheadrightarrow S^k\).
For an element \(g \in \mapsp{T}{M}\),
we denote by \(\mapsp[g]{T}{M}\) the component of \(\mapsp{T}{M}\) containing \(g\).

\newcommand{\wedgequot}{q}
For \(f\in\spheresp{k}{M}\)
and \(g\in\mapsp{T}{M}\),
we define a map \(f+g\in\mapsp{T}{M}\) as follows.
Fix an embedded \(k\)-disk around \(t_0\) in \(T\).
Then we have
the quotient map \(\wedgequot\colon T \to S^k\vee T\),
which is given by pinching the boundary of the embedded disk.
Since \(M\) is path-connected,
there is a map \(f'\in\spheresp{k}{M}\)
such that \(f'(s_0)=g(t_0)\) and \(f'\) is homotopic to \(f\) (without preserving base points).
Define \(f+g\) to be the composition
\begin{math}
  T \xrightarrow{\wedgequot} S^k\vee T \xrightarrow{f'\vee g} M.
\end{math}
Since \(M\) is 1-connected,
the map \(f+g\) is well-defined up to homotopy.

\newcommand{\inclsd}{\iota}
Instead of \cref{equation:STCopDiagram},
we consider the diagram
\begin{equation}
  \label{equation:newCopDiagram}
  \begin{tikzcd}
    \mapsp[f+g]{T}{M} \ar[d,"\res"]
    & \spheresp[f]{k}{M} \times_M \mapsp[g]{T}{M} \ar[l,"\comp"] \ar[r,"\incl"] \ar[d,"\pr_1"]
    & \spheresp[f]{k}{M} \times \mapsp[g]{T}{M} \\
    \disksp{k}{M} & \spheresp[f]{k}{M}, \ar[l,"\inclsd"]
  \end{tikzcd}
\end{equation}
where
the square is a pullback diagram,
the map \(\res\) is the restriction to the embedded \(k\)-disk, and
the map \(\inclsd\) is the inclusion induced by the quotient map \(D^k\to S^k\).

Note that
the above diagram is related to the diagram \cref{equation:STCopDiagram}
in the following way.
When \(M\) is \(k\)-connected,
we have the diagram
\begin{equation}
  \label{equation:relationDiagram}
  \begin{tikzcd}
    \mapsp{T}{M} \ar[d,"\res"]
    & \spheresp{k}{M} \times_M \mapsp{T}{M} \ar[l,"\comp"] \ar[d,"\pr_1"] \\
    \disksp{k}{M} \ar[d,"\res"] & \spheresp{k}{M} \ar[l,"\inclsd"] \ar[d,"\ev"]\\
    \spheresp{k-1}{M} & M, \ar[l,"\inclconst"]
  \end{tikzcd}
\end{equation}
where the two squares are pullback diagrams (and hence so is the outer square).
In this diagram,
the upper square coincides with \cref{equation:newCopDiagram}
and the outer square coincides with \cref{equation:STCopDiagram}.
We use this diagram
to compare the two brane coproducts in \cref{section:comparison}.

We define the dual
\begin{equation}
  \dcopnew\colon\cohom[*]{\spheresp[f]{k}{M} \times_M\mapsp[g]{T}{M}}
  \to \cohom[*+m]{\mapsp[f+g]{T}{M}}
\end{equation}
of the brane coproduct
by the composition
\begin{equation}
  \shriek\comp \circ \incl^*\colon
  \cohom{\mapsp[f+g]{T}{M}}
  \xrightarrow{\incl^*} \cohom{\spheresp[f]{k}{M} \times_M \mapsp[g]{T}{M}}
  \xrightarrow{\shriek\comp} \cohom[*+m]{\spheresp[f]{k}{M} \times \mapsp[g]{T}{M}}.
\end{equation}
Here,
\(\shriek\comp\) is the shriek map constructed from the diagram \cref{equation:newCopDiagram}.
In order to define it,
we need the corollary of the following proposition.
\todo{\cite[p.427, Lemma 1]{felix-thomas09}に載ってる．
  可換図式(\cref{proposition:diagonalClass}で使う)はないけど，どうだろう？}

\newcommand{\maptopd}{F}
\newcommand{\isomexthomol}{\Phi}
\newcommand{\exteval}{\Psi}
\newcommand{\partialeval}{{\mathcal E}}
\newcommand{\elmCohomX}{x}
\newcommand{\cohomdeg}{j}
\begin{proposition}[{\cite[Lemma 1]{felix-thomas09}}]
  \label{proposition:extPoincare}
  Let \(\maptopd\colon X\to N\) be a map between 0-connected spaces.
  Assume that \(N\) is a Poincaré duality space of dimension \(n\).
  Define a linear map
  \begin{equation}
    \isomexthomol\colon
    \ext^l_{\cochain{N}}(\cochain{X},\cochain{N})
    \to \hom_\K(\cohom[n-l]{X}, \cohom[n]{N})
  \end{equation}
  by
  \begin{math}
    \isomexthomol(\alpha) = \exttohom{\alpha}|_{\cohom[n-l]{X}}.
  \end{math}
  Then \(\isomexthomol\) is an isomorphism.

\end{proposition}

Then we have the following corollary,
which is an analogue of \cref{theorem:extSphereSpace}
for the case of the non-symmetric brane coproduct.

\begin{corollary}
  \label{corollary:extPoincare}
  Consider the same assumption with \cref{proposition:extPoincare}.
  Additionally assume \(l=n\) and \(\cohomdeg=0\).
  Then we have an isomorphism
  \begin{equation}
    \ext^n_{\cochain{N}}(\cochain{X},\cochain{N}) \xrightarrow{\cong} \cohom[n]{N};\quad
    \alpha \mapsto \exttohom{\alpha}(1).
  \end{equation}
\end{corollary}

Applying \cref{corollary:extPoincare} to the case \(\maptopd = \inclsd\) and \(n=m\),
we have the generator
\begin{equation}
  \shriek\inclsd
  \in \ext^m_{\cochain{\disksp{k}{M}}}(\cochain{\spheresp[f]{k}{M}},\cochain{\disksp{k}{M}})
  \cong \cohom[m]{\disksp{k}{M}} \cong \K
\end{equation}
up to non-zero scalar multiplication.
Using this element with the diagram \cref{equation:newCopDiagram},
we define
\begin{math}
  \shriek\comp = \exttohom{\extlift{\res}(\shriek\inclsd)}.
\end{math}
This completes the definition of the non-symmetric brane coprdouct.

\newcommand{\sphereact}{\rho}
Next we give more convenient description of \(\shriek\comp\).
Consider the commutative diagram
\begin{equation}
  \label{equation:convenientShriekCompDiagram}
  \begin{tikzcd}[row sep=2.5em]
    & \mapsp[f+g]{T}{M} \ar[dd,"\res" near end] \ar[dl,swap,"="]&&
    \spheresp[f]{k}{M} \times_M \mapsp[g]{T}{M}
    \ar[ll,"\comp"] \ar[dd,"\pr_1"] \ar[dl,"\sphereact","\simeq"'] \ar[r,"\incl"]
    & \spheresp[f]{k}{M} \times \mapsp[g]{T}{M} \\
    \mapsp[f+g]{T}{M} \ar[dd,swap,"\ev_{t_0}"] &&
    \spheresp[f]{k}{M} \times_M \mapsp[f+g]{T}{M} \ar[ll,crossing over,"\pr_2" near start,swap] \\
    & \disksp{k}{M} \ar[dl,swap,"\ev_{d_0}","\simeq"',swap] && \spheresp[f]{k}{M} \ar[ll,"\inclsd" near end] \ar[dl,"="]\\
    M &&
    \spheresp[f]{k}{M}, \ar[ll,"\ev_{s_0}"] \ar[uu,<-,crossing over,swap,"\pr_1" near end]
  \end{tikzcd}
\end{equation}
where the front and back square are pullback squares.
Here
\begin{math}
  \sphereact\colon
  \spheresp[f]{k}{M} \times_M \mapsp[g]{T}{M} \xrightarrow{\simeq}
  \spheresp[f]{k}{M} \times_M \mapsp[f+g]{T}{M}
\end{math}
is defined by
\(\sphereact(\varphi, \psi) = (\varphi, \varphi+\psi)\),
which is well-defined since we are working on the fiber product over \(M\).
By \cref{proposition:naturalityOfShriek},
we have
\begin{math}
  \exttohom{\extlift{\res}(\shriek\inclsd)} \circ \sphereact^*
  = \exttohom{\extlift{\ev_{t_0}}(\tilde{\shriek\inclsd})}
\end{math}
and hence
\begin{equation}
  \label{equation:convenientShriekComp}
  \shriek\comp = \exttohom{\extlift{\ev_{t_0}}(\tilde{\shriek\inclsd})} \circ (\sphereact^*)^{-1}.
\end{equation}
Here
\begin{math}
  \tilde{\shriek\inclsd} \in \ext^m_{\cochain{M}}(\cochain{\spheresp[f]{k}{M}}, \cochain{M})
\end{math}
is the image of \(\shriek\inclsd\)
under the isomorphism induced by \(\ev_{d_0}\).

\section{Computation of the non-symmetric brane coproduct}
\label{section:computeNewCoprod}
\newcommand{\const}{0}

In this section,
we use the same notation and assumptions as in \cref{section:newCoprodConstruction}.
Let
\begin{math}
  \const \in \spheresp{k}{M}
\end{math}
be the constant map and
denote
the orientation class of \(M\) by
\begin{math}
  \omega\in\cohom[m]{M}.
\end{math}
This section is devoted to the proof of
the following formula of the non-symmetric brane coproduct.

\newcommand{\cpleft}{u}
\newcommand{\cpright}{v}
\begin{theorem}
  \label{theorem:computeNewCoprod}
  For the case \(f=\const\in \htpyset{S^k}{M}\),
  the non-symmetric coproduct
  \begin{equation}
    \dcopnew\colon \cohom{\spheresp[\const]{k}{M}\times\mapsp[g]{T}{M}}
    \to \cohom[*+m]{\mapsp[g]{T}{M}}
  \end{equation}
  is described by
  \begin{equation}
    \dcopnew(\cpleft\times\cpright) = \ev_{t_0}^*(\omega\cdot\inclconst^*(\cpleft))\cdot\cpright,
  \end{equation}
  where
  \(\cpleft\times\cpright\)
  denotes the cross product of
  \(\cpleft\in\cohom{\spheresp[\const]{k}{M}}\)
  and
  \(\cpright\in\cohom{\mapsp[g]{T}{M}}\),
  and
  \begin{math}
    \inclconst\colon M \to \spheresp[\const]{k}{M}
  \end{math}
  is the embedding as constant maps.
\end{theorem}
\todo{\(\cpleft\) と \(\cpright\) の記号が被ってないか確認}

This is an analogue
of 
\cite[Theorem 30]{menichi13}
in the case of the non-symmetric coproduct.
Note that, when \(\chara\K=0\),
the above formula can be proved easily
by using rational models of mapping spaces given in \cite{berglund15}.

\newcommand{\sectfiberprod}{\sigma}
To prove \cref{theorem:computeNewCoprod},
we need some propositions.
First we investigate the map
\((\sphereact^*)^{-1}\)
in \cref{equation:convenientShriekComp}.
Define
\begin{math}
  \sectfiberprod\colon \mapsp[g]{T}{M}
  \to \spheresp[0]{k}{M} \times_M \mapsp[g]{T}{M}
\end{math}
by
\(\sectfiberprod(\psi) = (\inclconst(\psi(t_0)), \psi)\).

\newcommand{\elmFiberProd}{x}
\begin{proposition}
  \label{proposition:computeSphereAction}
  For any \(\elmFiberProd \in \cohom{\spheresp[0]{k}{M} \times_M \mapsp[g]{T}{M}}\),
  we have
  \begin{equation}
    (\sphereact^*)^{-1}\elmFiberProd - \elmFiberProd
    \in \ker(\sectfiberprod^*).
  \end{equation}
\end{proposition}
\begin{proof}
  Let \(\bar\sphereact\) be the homotopy inverse of \(\sphereact\).
  Then we have \(\sphereact\circ\sectfiberprod \simeq \sectfiberprod\) and hence
  \begin{math}
    \sectfiberprod^*((\sphereact^*)^{-1}\elmFiberProd - \elmFiberProd)
    = \sectfiberprod^*(\bar\sphereact^*\elmFiberProd - \elmFiberProd)
    = \sectfiberprod^*\elmFiberProd - \sectfiberprod^*\elmFiberProd
    = 0.
  \end{math}
\end{proof}

Next we relate
\(\ker(\sectfiberprod^*)\)
with
\(\exttohom{{\extlift{\ev_{t_0}}\tilde{\shriek\inclsd}}}\).

\begin{proposition}
  \label{proposition:kerSubsetKer}
  Consider a pullback diagram
  \begin{equation}
    \begin{tikzcd}
      E \ar[d,"p"] & X \ar[l,"g"] \ar[d,"q"]\\
      B & A \ar[l,"f"]
    \end{tikzcd}
  \end{equation}
  such that the Eilenberg-Moore map is an isomorphism,
  and take an element
  \begin{math}
    \alpha \in \ext_{\cochain{B}}(\cochain{A}, \cochain{B}).
  \end{math}
  Let
  \(\sigma\colon E\to X\)
  and
  \(\tau\colon B\to A\)
  be sections of \(g\) and \(f\), respectively,
  satisfying \(q\circ\sigma = \tau\circ p\).
  Assume that
  there is an element
  \begin{math}
    \tilde\alpha \in \ext_{\cochain{B}}(\cochain{B}, \cochain{B})
  \end{math}
  which is mapped to \(\alpha\) by the map induced by \(\tau\).
  Then
  \begin{equation}
    \ker(\sectfiberprod^*) \subset \ker(\exttohom{\extlift{p}\alpha}).
  \end{equation}
\end{proposition}

\begin{proof}
  Applying \cref{proposition:naturalityOfShriek} to the following diagram,
  we have
  \begin{math}
    \exttohom{\extlift{p}\alpha} = \exttohom{\extlift{p}\tilde\alpha} \circ \sectfiberprod^*,
  \end{math}
  and this proves the proposition.
  \begin{equation}
    \begin{tikzcd}[row sep=2.5em]
      & E \ar[dd,"p" near end] \ar[dl,swap,"="]&&
      E \ar[ll,"="] \ar[dd,"p"] \ar[dl,"\sigma"]\\
      E \ar[dd,swap,"p"] &&
      X \ar[ll,crossing over,"g" near start,swap] \\
      & B \ar[dl,swap,"=",swap] && B \ar[ll,"=" near end] \ar[dl,"\tau"]\\
      B &&
      A \ar[ll,"f"] \ar[uu,<-,crossing over,swap,"q" near end]
    \end{tikzcd}
  \end{equation}
\end{proof}

Next, we consider the diagram
\begin{equation}
  \begin{tikzcd}
    \mapsp[g]{T}{M} \ar[d,swap,"\ev_{t_0}"] &
    \spheresp[\const]{k}{M} \times_M \mapsp[g]{T}{M} \ar[l,"\pr_2",swap] \\
    M &
    \spheresp[\const]{k}{M}. \ar[l,"\ev_{s_0}"] \ar[u,<-,swap,"\pr_1"]
  \end{tikzcd}
\end{equation}
Note that the maps
\begin{math}
  \sectfiberprod\colon \mapsp[g]{T}{M}
  \to \spheresp[0]{k}{M} \times_M \mapsp[g]{T}{M}
\end{math}
and
\(\inclconst\colon M \to \spheresp[\const]{k}{M}\),
are sections of \(\pr_2\) and \(\ev_{s_0}\), respectively.
Recall from \cref{equation:convenientShriekComp} that
we are using
\begin{math}
  \tilde{\shriek\inclsd} \in \ext^m_{\cochain{M}}(\cochain{\spheresp[\const]{k}{M}}, \cochain{M})
\end{math}
to compute the non-symmetric brane coproduct.
\begin{corollary}
  \label{corollary:kerSubsetKer}
  Under the above notation, we have
  \begin{equation}
    \ker(\sectfiberprod^*) \subset \ker\big(\exttohom{\extlift{\ev_{t_0}}(\tilde{\shriek\inclsd})}\big).
  \end{equation}
\end{corollary}
\begin{proof}
  By \cref{corollary:extPoincare},
  the map \(\inclconst\) induces an isomorphism
  \begin{equation}
    \ext^m_{\cochain{M}}(\cochain{M},\cochain{M})
    \xrightarrow{\cong}
    \ext^m_{\cochain{M}}(\cochain{\spheresp[\const]{k}{M}},\cochain{M}).
  \end{equation}
  Thus we obtain \(\tilde\alpha\) as in the assumption of \cref{proposition:kerSubsetKer},
  and hence it proves the corollary.
\end{proof}

By \cref{proposition:computeSphereAction} and \cref{corollary:kerSubsetKer},
\cref{theorem:computeNewCoprod} reduces to the following simple proposition.

\begin{proposition}
  \label{proposition:computeOnCrossProduct}
  Consider a pullback diagram
  \begin{equation}
    \begin{tikzcd}
      E \ar[d,"p"] & X \ar[l,"g"] \ar[d,"q"]\\
      B & A \ar[l,"f"]
    \end{tikzcd}
  \end{equation}
  such that the Eilenberg-Moore map is an isomorphism,
  and an element
  \begin{math}
    \alpha \in \ext_{\cochain{B}}(\cochain{A}, \cochain{B}).
  \end{math}
  Then the composition
  \begin{math}
    \cohom{A\times E} \xrightarrow{\incl^*}
    \cohom{X} \xrightarrow{\exttohom{\extlift{p}\alpha}}
    \cohom{E}
  \end{math}
  satisfies
  \begin{equation}
    \exttohom{\extlift{p}\alpha} \circ \incl^* (\cpleft\times\cpright)
    = p^*(\exttohom{\alpha}(\cpleft))\cdot\cpright
  \end{equation}
  for any
  \(\cpleft\in\cohom{A}\)
  and
  \(\cpright\in\cohom{E}\).
\end{proposition}

\begin{proof}
  Consider the diagram
  \begin{equation}
    \begin{tikzcd}[row sep=2.5em]
      & E \ar[dd,"p" near end] \ar[dl,swap,"{(p,\id)}"]&&
      X \ar[ll,"g"] \ar[dd,"q"] \ar[dl,"{(q,g)}" near end,swap]\\
      B\times E \ar[dd,swap,"\pr_1"] &&
      A\times E \ar[ll,crossing over,"f\times\id" near start,swap] \\
      & B \ar[dl,swap,"=",swap] && A \ar[ll,"f" near end] \ar[dl,"="]\\
      B &&
      A. \ar[ll,"f"] \ar[uu,<-,crossing over,swap,"\pr_1" near end]
    \end{tikzcd}
  \end{equation}
  By \cref{proposition:naturalityOfShriek}, we have
  \begin{equation}
    \exttohom{\extlift{p}\alpha} \circ \incl^*
    = (p,\id)^*\circ(\exttohom{\extlift{\pr_1}\alpha}).
  \end{equation}
  Since the fibration \(\pr_1\) is very simple,
  we can prove
  \begin{equation}
    \exttohom{\extlift{\pr_1}\alpha}(\cpleft\times\cpright)
    = \exttohom{\alpha}(\cpleft)\times\cpright
  \end{equation}
  by a direct computation from the definition.
\end{proof}

Now we give a proof of \cref{theorem:computeNewCoprod}
using the above corollary and propositions.

\begin{proof}[Proof of \cref{theorem:computeNewCoprod}]
  By \cref{equation:convenientShriekComp},
  we have
  \begin{equation}
    \dcopnew(\cpleft\times\cpright) =
    \exttohom{\extlift{\ev_{t_0}}(\tilde{\shriek\inclsd})} \circ (\sphereact^*)^{-1}
    \circ \incl^* (\cpleft\times\cpright).
  \end{equation}
  By \cref{proposition:computeSphereAction} and \cref{corollary:kerSubsetKer},
  we have
  \begin{equation}
    \exttohom{\extlift{\ev_{t_0}}(\tilde{\shriek\inclsd})} \circ (\sphereact^*)^{-1}
    = \exttohom{\extlift{\ev_{t_0}}(\tilde{\shriek\inclsd})}.
  \end{equation}
  Thus
  \begin{equation}
    \dcopnew(\cpleft\times\cpright)
    = \exttohom{\extlift{\ev_{t_0}}(\tilde{\shriek\inclsd})} \circ \incl^* (\cpleft\times\cpright),
  \end{equation}
  and hence \cref{proposition:computeOnCrossProduct} proves the theorem.
\end{proof}

\section{Comparison of two brane coproducts}
\label{section:comparison}
In this section,
we compare the two brane coproducts.
As an application,
we prove \cref{theorem:vanishingMainTheorem}.

\subsection{Proof of \cref{theorem:vanishingGeneral}}
In this subsection,
we prove \cref{theorem:vanishingGeneral}.

Let \(\K\) be a field of any characteristic,
\(k\) a positive integer,
and \(M\) a \(k\)-connected Poincaré duality space of dimension \(m\).
We fix an arbitrary element
\begin{equation}
  \myshriek\in
  \ext^{m}_{\cochain{\spheresp{k-1}{M}}}(\cochain{M}, \cochain{\spheresp{k-1}{M}}).
\end{equation}
Then we have the brane coproduct
\begin{equation}
  \dcop_\myshriek\colon
  \cohom{\spheresp{k}{M}\times\spheresp{k}{M}}
  \to \cohom[*+m]{\spheresp{k}{M}}
\end{equation}
for the case \(S = T = S^k\)
by the construction given in \cref{section:reviewPreviousBraneCoproduct}.

\begin{remark}
  The degree \(m\) of the element \(\myshriek\) is different from
  the degree \(\bar{m}\) of \(\shriek\inclconst\) in \cref{theorem:extSphereSpace}.
  These degrees coincide
  under the assumption
  \cref{item:assump_brane} of \cref{theorem:vanishingMainTheorem}
  (see \cref{remark:assumpIsNatural}).
  This case will be treated in \cref{subsection:comparison_higher}
  and \cref{section:modelOfShriek}.
\end{remark}

To compare \(\dcop_\myshriek\) with \(\dcopnew\),
we relate \(\myshriek\) with \(\shriek\inclsd\).
As in \cref{theorem:vanishingGeneral},
define \(\lambda_\myshriek\in \K\) by the equation
\begin{equation}
  \inclconst^*\circ(\exttohom{\myshriek})(1) = \lambda_\myshriek \omega
  \in \cohom[m]{M},
\end{equation}
where \(\omega\) is the orientation class of \(M\).

\begin{proposition}
  \label{proposition:compareShriek}
  Under the above notation,
  we have
  \begin{equation}
    \extlift{\res}(\myshriek) = \lambda_\myshriek \shriek\inclsd
    \in
    \ext^m_{\cochain{\disksp{k}{M}}}(\cochain{\spheresp{k}{M}}, \cochain{\disksp{k}{M}}),
  \end{equation}
  where \(\extlift{\res}\) is the lift
  along the lower pullback square in \cref{equation:relationDiagram}.
  Moreover, this implies
  \begin{equation}
    \dcop_\myshriek = \lambda_\myshriek \dcopnew\colon
    \cohom{\spheresp{k}{M}\times\spheresp{k}{M}}
    \to \cohom{\spheresp{k}{M}}.
  \end{equation}
\end{proposition}
\begin{proof}
  Let \(\omega\in\cohom[m]{M}\cong\cohom[m]{\disksp{k}{M}}\) be the orientation class.
  Recall from \cref{corollary:extPoincare} that
  \(\shriek\inclsd\) is characterized by
  \(\exttohom{\shriek\inclsd}(1) = \omega\).
  Hence it is enough to prove
  \begin{math}
    \exttohom{\extlift{\res}(\myshriek)}(1) = \lambda_\myshriek\omega.
  \end{math}

  \newcommand{\unitcocycle}{u}
  \newcommand{\shriekdiagrep}{\varphi}
  Let
  \begin{math}
    \eta\colon P\xrightarrow{\qis}\cochain{M}
  \end{math}
  be a semifree resolution of \(\cochain{M}\) over \(\cochain{\spheresp{k-1}{M}}\),
  and \(\unitcocycle\in P\) a cocycle such that \(\eta(\unitcocycle) = 1\).
  Take a representative
  \begin{math}
    \shriekdiagrep\in \hom_{\cochain{\spheresp{k-1}{M}}}(P, \allowbreak \cochain{\spheresp{k-1}{M}})
  \end{math}
  of \(\myshriek\).
  Then we have
  \begin{math}
    [\shriekdiagrep(u)] = \exttohom{\myshriek}(1)
    \in \cohom[m]{\spheresp{k-1}{M}}.
  \end{math}
  By definition,
  \(\exttohom{\extlift{\res}(\myshriek)}\)
  is represented by the chain map
  \(\id_{\cochain{\disksp{k}{M}}}\otimes\shriekdiagrep\)
  in
  \begin{equation}
    \hom_{\cochain{\disksp{k}{M}}}
    (\cochain{\disksp{k}{M}}\otimes_{\cochain{\spheresp{k-1}{M}}}P,
    \cochain{\disksp{k}{M}}\otimes_{\cochain{\spheresp{k-1}{M}}}\cochain{\spheresp{k-1}{M}}).
  \end{equation}
  Hence we have
  \begin{math}
    \exttohom{\extlift{\res}(\myshriek)}(1)
    = [(\id_{\cochain{\disksp{k}{M}}}\otimes\shriekdiagrep)(1\otimes\unitcocycle)]
    = \inclconst^*[\shriekdiagrep(\unitcocycle)]
    = \lambda_\myshriek\omega
    \in \cohom[m]{M}
  \end{math}
  under the identification \(\cohom[m]{M} = \cohom[m]{\disksp{k}{M}}\).
  This proves the proposition.
\end{proof}

Next we consider the commutativity of the coproduct \(\dcop_\myshriek\).
Let
\(\orirev\colon \spheresp{k-1}{M} \to \spheresp{k-1}{M}\)
be the map induced from the orientation reversing map on \(S^{k-1}\),
satisfying \(\orirev^2=\id\).
Then \(\orirev\) induces the map
\begin{align}
  \orirev^* = \ext_{\orirev^*}(\id, \orirev^*) \colon
  &\ext_{\cochain{\spheresp{k-1}{M}}}(\cochain{M}, \cochain{\spheresp{k-1}{M}}) \\
  &\to
  \ext_{\cochain{\spheresp{k-1}{M}}}(\cochain{M}, \cochain{\spheresp{k-1}{M}}).
\end{align}
By the definition of \(\lambda_\myshriek\),
we have
\begin{equation}
  \label{equation:orirevNotAffect}
  \lambda_\myshriek = \lambda_{\orirev^*\myshriek}.
\end{equation}

The coproduct is commutative in the following sense.
\todo{\(\alpha, \beta\) じゃなくて \(\cpleft, \cpright\) を使う？ここ以外のもcheck}
\begin{proposition}
  \label{proposition:commutativityOfPreviousCoproduct}
  \begin{equation}
    \dcop_\myshriek(\alpha \times \beta)
    = (-1)^{\deg{\alpha}\deg{\beta}}\dcop_{\orirev^*\myshriek}(\beta\times\alpha)
  \end{equation}
\end{proposition}
The proposition is proved by the same method with
the commutativity of the brane coproduct \(\dcop_{\shriek\inclconst}\)
\cite[Theorem 1.5]{wakatsuki18:toappear}.
Note that we used the equation
\(\orirev^*\shriek\inclconst = (-1)^{\bar{m}}\shriek\inclconst\)
\cite[Equation (7.11)]{wakatsuki18:toappear}
to prove
\begin{math}
  \dcop_{\shriek\inclconst}(\alpha \times \beta)
  = (-1)^{\deg{\alpha}\deg{\beta}+\bar{m}}\dcop_{\shriek\inclconst}(\beta\times\alpha).
\end{math}

Using the above propositions,
we give a proof of \cref{theorem:vanishingGeneral}.

\begin{proof}[Proof of \cref{theorem:vanishingGeneral}]
  Since the fibration \(\ev_0\colon \spheresp{k}{M} \to M\)
  has a section \(\inclconst\colon M\to\spheresp{k}{M}\),
  we have a decomposition
  \(\cohom[>0]{\spheresp{k}{M}} \cong \cohom[>0]{M} \oplus \ker(\inclconst^*)\).
  When \(\alpha \in \cohom[>0]{M}\),
  we have \(\alpha\omega = 0 \in \cohom[\deg{\alpha + m}]{M}=0\).
  Hence we assume \(\alpha \in \ker(\inclconst^*)\).
  Then, by \cref{theorem:computeNewCoprod},
  we have
  \begin{align}
    \dcopnew(\alpha\times 1) &= \ev_0^*(\omega\cdot\inclconst^*(\alpha))\cdot 1 = 0 \\
    \dcopnew(1\times \alpha) &= \ev_0^*(\omega\cdot\inclconst^*(1))\cdot \alpha
                       = \ev_0^*\omega\cdot\alpha.
  \end{align}
  Moreover, we have
  \begin{equation}
    \lambda_\myshriek\dcopnew(\alpha\times 1)
    = \dcop_\myshriek(\alpha\times 1)
    = \pm\dcop_{\orirev^*\myshriek}(1\times\alpha)
    = \pm\lambda_\myshriek\dcopnew(1\times\alpha)
  \end{equation}
  by \cref{equation:orirevNotAffect},
  \cref{proposition:commutativityOfPreviousCoproduct}, and
  \cref{proposition:compareShriek}.
  These equations prove the theorem.
\end{proof}

\subsection{Proof of \cref{theorem:vanishingMainTheorem} \cref{item:assump_loop}}
In this subsection,
we prove \cref{theorem:vanishingMainTheorem}
under the assumption \cref{item:assump_loop}.
As a preparation of the proof,
we investigate the map \(\isomexthomol\) in \cref{proposition:extPoincare}.

\newcommand{\pair}[3][]{\langle{#2},{#3}\rangle_{#1}}
As in \cref{proposition:extPoincare},
let \(X\) be a 0-connected space,
\(N\) a Poincaré duality space of dimension \(n\), and
\(\maptopd\colon X\to N\) a map.
We denote the orientation class of \(N\) by \(\omega_N\in\cohom[n]{N}\)
and the fundamental class by \([N]\in\homol[n]{N}\).
Then we have
\(\pair[N]{\omega_N}{[N]} = 1\),
where
\begin{math}
  \pair[N]{-}{-}\colon
  \cohom{N}\otimes\homol{N} \to \K
\end{math}
denotes the pairing.

\newcommand{\cyclex}{x}
\newcommand{\cocyclex}{\nu}
\newcommand{\cocyclen}{\psi}
\begin{proposition}
  \label{proposition:evaluateExtPoincare}
  Fix arbitrary elements \(\cyclex\in\homol[n-l]{X}\)
  and \(\cocyclex\in\cohom[j]{X}\).
  Let
  \begin{math}
    \beta_\cyclex\colon
    \cohom[n-l]{X} \to \cohom[n]{N}
  \end{math}
  be the linear map defined by
  \(\beta_\cyclex(\varphi) = \pair[X]{\varphi}{\cyclex}\omega_N\)
  for \(\varphi\in\cohom[n-l]{X}\).
  Using the isomorphism \(\isomexthomol\) in \cref{proposition:extPoincare},
  we define
  \begin{equation}
    \alpha_{\cyclex} = \isomexthomol^{-1}(\beta_\cyclex)
    \in \ext^l_{\cochain{N}}(\cochain{X},\cochain{N}).
  \end{equation}
  Then
  the element
  \(\exttohom{\alpha_{\cyclex}}(\cocyclex)\in\cohom[l+j]{N}\)
  is the unique element which satisfies
  \begin{equation}
    \label{equation:evaluateExtPoincare}
    \pair[N]{\cocyclen}{\ \exttohom{\alpha_{\cyclex}}(\cocyclex)\cap[N]}
    = (-1)^{l(n-l-j)}\pair[X]{\maptopd^*\cocyclen\cdot\cocyclex}{\ \cyclex}
  \end{equation}
  for any \(\cocyclen\in\cohom[n-l-j]{N}\).
\end{proposition}
\begin{proof}
  Since the cap product \(-\cap[N]\) is an isomorphism by the Poincaré duality,
  such element is uniquely determined.
  Since \(\exttohom{\alpha_\cyclex}\) is \(\cohom{N}\)-linear,
  we have
  \begin{math}
    \cocyclen\cdot\exttohom{\alpha_\cyclex}(\cocyclex)
    = (-1)^{l(n-l-j)}\exttohom{\alpha_\cyclex}(\maptopd^*\cocyclen\cdot\cocyclex).
  \end{math}
  Using this equation,
  we can prove \cref{equation:evaluateExtPoincare} by a straightforward calculation.
\end{proof}

Now we begin the proof of \cref{theorem:vanishingMainTheorem} \cref{item:assump_loop}.
Let \(M\) be a 1-connected Poincaré duality space of dimension \(m\).
Here we write
\(LM = \spheresp{1}{M}\) and
\(\diag = \inclconst\colon M\to M\times M\)
as usual.
Recall that
\begin{equation}
  \shriek\diag \in
  \ext^m_{\cochain{M\times M}}(\cochain{M}, \cochain{M\times M}) \cong \K
\end{equation}
is the generator,
which is defined up to non-zero scalar multiplication.

\begin{proposition}
  \label{proposition:diagonalClass}
  The element
  \begin{math}
    \exttohom{\shriek\diag}(1)
    \in \cohom[m]{M\times M}
  \end{math}
  is the diagonal class,
  i.e.\ the Poincaré dual of the homology class
  \(\diag_*[M]\in\cohom[m]{M\times M}\).
  In particular, we have
  \begin{equation}
    \diag^*\circ(\exttohom{\shriek\diag})(1) = \chi(M)\omega \in \cohom{M}.
  \end{equation}
\end{proposition}
\begin{proof}
  Since \(M\times M\) is also a Poincaré duality space,
  we can apply \cref{proposition:evaluateExtPoincare} for the case
  \(\maptopd = \diag\), \(n = 2m\), \(l=m\),  \(j=0\),
  \(\cyclex = [M]\), and \(\cocyclex = 1\).
  Since \(\shriek\diag\) is defined up to non-zero scalar multiplication,
  we may assume
  \(\shriek\diag = (-1)^m\alpha_{[M]}\).
  By \cref{equation:evaluateExtPoincare}, we have
  \begin{equation}
    \pair[M^2]{\cocyclen}{\ \exttohom{\shriek\diag}(1)\cap[M^2]}
    = \pair[M]{\diag^*\cocyclen\cdot 1}{\ [M]}
    = \pair[M^2]{\cocyclen}{\ \diag_*[M]}
  \end{equation}
  for any
  \(\cocyclen \in \cohom[m]{M^2}\),
  and hence
  \begin{math}
    \exttohom{\shriek\diag}(1)\cap[M^2]
    = \diag_*[M].
  \end{math}

  It is well-known that the diagonal class satisfies the required property
  (c.f.\ e.g.\ \cite[pp.~127--129, Section 11]{milnor-stasheff}).
\end{proof}

Now we have the following theorem using the above lemma.

\begin{theorem}[\cref{theorem:vanishingMainTheorem} \cref{item:assump_loop}]
  \label{theorem:vanishingOnLoopSpace}
  Let \(M\) be a 1-connected Poincaré duality space over \(\K\)
  and denote its orientation class by \(\omega\in\cohom[m]{M}\).
  Then, for any \(\alpha\in\cohom[>0]{LM}\), we have
  \begin{equation}
    \chi(M)\ev_0^*\omega\cdot\alpha = 0 \in \cohom[\deg{\alpha}+m]{LM}.
  \end{equation}
\end{theorem}
\begin{proof}
  Apply \cref{theorem:vanishingGeneral} and \cref{proposition:diagonalClass}.
\end{proof}

\begin{remark}
  This theorem generalizes \cite[Theorem 1]{menichi13}
  in the sense that our theorem can be applied to Poincaré duality spaces,
  not only manifolds.
  \todo{\(\ext\)も触れる？}
\end{remark}

\subsection{Proof of \cref{theorem:vanishingMainTheorem} \cref{item:assump_brane}}
\label{subsection:comparison_higher}
In this section,
we prove \cref{theorem:vanishingMainTheorem} under the assumption \cref{item:assump_brane}.

Let \(k\) be a positive \textit{odd} integer and
\(M\) a \(k\)-connected Poincaré duality space over \(\K\) of dimension \(m\).
Assume \(\chara\K=0\) and
\(\dim_\K\left(\bigoplus_n\pi_n(M)\otimes\K\right) < \infty\).

First we explain why we assume \(k\) is odd
in the assumption \cref{item:assump_brane} in \cref{theorem:vanishingMainTheorem}.

\begin{remark}
  \label{remark:assumpIsNatural}
  Let \(x_1,\ldots, x_p\) and \(y_1,\ldots,y_q\) be
  bases of
  \(\bigoplus_n\pi_{2n}(M)\otimes\K\) and
  \(\bigoplus_n\pi_{2n-1}(M)\otimes\K\),
  respectively.
  Then we have the following.
  \begin{itemize}
    \item \(\chi(M)\neq 0\) if and only if \(p=q\).
      See \cref{theorem:elliptic} for details.
    \item Define \(a_i = \deg{x_i}\) and \(b_j = \deg{y_j}\).
      By \cite[Proposition 5.2]{felix-halperin-thomas88}, we have
      \(m = \dim M = \sum_jb_j + \sum_i(1-a_i)\).
      By the same formula, we have
      \begin{equation}
        \bar{m}  = \dim\Omega^{k-1}M =
        \begin{cases}
          m-(q-p)(k-1), & \text{if \(k\) is odd,} \\
          -m-(k-2)p+kq, & \text{if \(k\) is even.}
        \end{cases}
      \end{equation}
      Thus, except for rare exceptions,
      \(\bar{m}\) coincides with \(m\)
      if and only if \(k\) is odd and \(p = q\).
  \end{itemize}
  Since the statement of \cref{theorem:vanishingMainTheorem} is trivial when \(\chi(M)=0\),
  we are interested only in the case \(\chi(M)\neq 0\), i.e.\ \(p=q\).
  Moreover, since we will compare two brane coproducts,
  their degrees \(m\) and \(\bar{m}\) must coincide.
  Hence we may assume \(k\) is odd.
  This explains why
  the assumption \cref{item:assump_brane} in \cref{theorem:vanishingMainTheorem}
  is natural one.
\end{remark}


Now we give a proposition,
which is a key to prove \cref{theorem:vanishingMainTheorem} \cref{item:assump_brane}.

\begin{proposition}
  \label{proposition:shriekInclconst_EulerChar}
  Under the assumption \cref{item:assump_brane} in \cref{theorem:vanishingMainTheorem},
  there exists an element
  \begin{math}
    \myshriek \in
    \ext^{\bar{m}}_{\cochain{\spheresp{k-1}{M}}}(\cochain{M}, \cochain{\spheresp{k-1}{M}})
  \end{math}
  such that
  \begin{equation}
    \inclconst^*\circ(\exttohom{\myshriek})(1) = \chi(M)\omega \in \cohom{M}.
  \end{equation}
\end{proposition}
We defer the proof of the proposition to \cref{section:modelOfShriek}.
Applying the proposition and \cref{theorem:vanishingGeneral},
we have \cref{item:assump_brane} of \cref{theorem:vanishingMainTheorem}.

\begin{theorem}[\cref{theorem:vanishingMainTheorem} \cref{item:assump_brane}]
  \label{theorem:vanishingOnSphereSpace}
  Under the assumption \cref{item:assump_brane} in \cref{theorem:vanishingMainTheorem},
  we have
  \begin{equation}
    \chi(M)\ev_0^*\omega\cdot\alpha = 0 \in \cohom[\deg{\alpha}+m]{\spheresp{k}{M}}
  \end{equation}
  for any \(\alpha\in\cohom[>0]{\spheresp{k}{M}}\).
\end{theorem}

Hence the rest of this article is
devoted to the proof of \cref{proposition:shriekInclconst_EulerChar}.

\section{Models of shriek maps}
\label{section:modelOfShriek}
In this section,
we give a proof of \cref{proposition:shriekInclconst_EulerChar}.
As a preparation of the proof,
we explicitly construct a model of the shriek map \(\shriek\inclconst\)
when the coefficient is a field \(\K\) of characteristic zero.
By \cref{equation:shriekInExt},
it is enough to construct a non-trivial element in
\begin{math}
  \ext^{\bar{m}}_{\cochain{\spheresp{k-1}{M}}}(\cochain{M}, \cochain{\spheresp{k-1}{M}}) \cong \K.
\end{math}
In \cref{subsection:modelOfShriek_candidate},
we construct a candidate of the shriek map,
whose non-triviality is proved in \cref{subsection:modelOfShriek_pure}
under some assumptions.

The construction is a generalization of the ones in \cite{naito13} and \cite{wakatsuki16},
which treat only the case \(k=1\).
Note that, in \cite[Proposition 6.2]{wakatsuki18:toappear},
the shriek map is explicitly constructed
when \(k\) is even and the minimal Sullivan model is pure,
which is much simpler than the one in this section.

Throughout this section,
we assume \(\chara\K = 0\) and
make full use of rational homotopy theory.
See \cite{felix-halperin-thomas01} for basic definitions and theorems.

For a graded vector space \(V\),
we define a graded vector space \(\susp[k]{V}\)
by \((\susp[k]{V})^n = V^{n+k}\).
For an element \(v\in V\),
we denote the corresponding element by \(\susp[k]{v}\in\susp[k]{V}\).
For simplicity, we write \(\susp{V} = \susp[1]{V}\).

Let \((\wedge V, d)\) be a Sullivan algebra satisfying \(\dim V < \infty\) and \(V^1 = 0\).
\newcommand{\base}{z}
\newcommand{\bi}{t}             
We fix a basis
\(\base_1, \ldots ,\base_r\) of \(V\)
such that
\(d\base_{\bi+1} \in \wedge V(\bi)\),
where
\(V(\bi) = \spanv_\K\{\base_1, \ldots, \base_\bi\}\).

\subsection{Construction of a chain map}
\label{subsection:modelOfShriek_candidate}
In this subsection,
we give an explicit construction of a candidate of the shriek map for \(k\geq 1\).
The construction is completely analogous to the one in \cite{wakatsuki16}.

\newcommand{\mdisk}[1]{{\mathcal D}^{#1}}
\newcommand{\msphere}[1]{{\mathcal S}^{#1}}
\newcommand{\mainterm}{\sigma}
\newcommand{\errterm}{\tau}
In this subsection,
we assume \(V^{\leq k} = 0\) additionally.
Write
\(\msphere{k-1} = \msphere{k-1}V = \wedge V \otimes \wedge \susp[k-1]{V}\)
and
\(\mdisk{k} = \mdisk{k}V = \wedge V \otimes \wedge \susp[k-1]{V} \otimes \susp[k]{V}\).
Here we define two Sullivan algebras \((\msphere{k-1}, d)\) and \((\mdisk{k}, d)\),
and two linear maps
\(\mainterm\colon V\to \msphere{k-1}\)
and
\(\errterm\colon V\to \mdisk{k}\).
Note that \((\msphere{k-1}, d)\) and \((\mdisk{k}, d)\)
are models of \(\spheresp{k-1}{M}\) and \(\disksp{k}{M}\), respectively.

Let \(\tilde{s}^{k-1}\colon \msphere{k-1}\to\msphere{k-1}\) be the derivation defined by
\(\tilde{s}^{k-1}(v) = \susp[k-1]{v}\) and \(\tilde{s}^{k-1}(\susp[k-1]{v}) = 0\).
By an abuse of notation,
we write \(\tilde{s}^{k-1}\) simply by \(\susp[k-1]{}\).
Similarly we define the derivation
\(\susp[k]{}\colon \mdisk{k}\to\mdisk{k}\).
Note that these derivations are not equal to the compositions of \(\susp[1]{}\)
(e.g.\ \(\susp[k-1]{}\neq \susp[1]{}\circ\cdots\susp[1]{}\)).

First we define the differentials \(d\) on \(\msphere{k-1}\) and \(\mdisk{k}\) in the case \(k=1\).
Then \((\msphere{0}, d)\) is just the tensor product \((\wedge V, d)\tpow2\).
The dga \((\mdisk{1}, d)\) is a relative Sullivan algebra over \((\wedge V, d)\tpow2\),
defined by the formula
\begin{math}
  d(\susp \base_\bi)
  = 1\otimes\base_\bi - \base_\bi\otimes 1
  - \sum_{n=1}^\infty\frac{(sd)^n}{n!}(\base_\bi\otimes 1)
\end{math}
inductively on \(\bi\)
(see \cite[Section 15 (c)]{felix-halperin-thomas01} or \cite[Appendix A]{wakatsuki16} for details).
Then, for \(v \in V\), we set
\(\mainterm v = 1\otimes v - v\otimes 1\)
and
\(\errterm v = - \sum_{n=1}^\infty\frac{(sd)^n}{n!}(v\otimes 1)\),
which satisfy
\(d\susp v = \mainterm v + \errterm v\).

Next we consider the case \(k\geq 2\).
Define the differential \(d\) on \(\msphere{k-1}\)
by the formula \(d\susp[k-1]v = (-1)^{k-1}\susp[k-1]{dv}\).
Set
\(\mainterm v = \susp[k-1]{v}\),
\(\errterm v = (-1)^k\susp[k]{dv}\).
Then we define the relative Sullivan algebra \((\mdisk{k}, d)\) over
\((\msphere{k-1}, d)\) by the formula
\(d\susp[k]{v} = \mainterm v + \errterm v\).
See \cite[Section 5]{wakatsuki18:toappear} for details.
By the following proposition,
we can use \(\msphere{k-1}\) and \(\mdisk{k}\) to construct
the shriek map \(\shriek\inclconst\).

\begin{proposition}[{\cite[Proposition 5.1]{wakatsuki18:toappear}}]
  Let \(M\) be a \(k\)-connected space
  and \((\wedge V, d)\) be its Sullivan model.
  Then the above algebras \(\msphere{k-1}\) and \(\mdisk{k}\) are
  Sullivan models of \(\spheresp{k-1}{M}\) and \(\disksp{k}{M}\).
  In particular, we have
  \begin{equation}
    \ext_{\cochain{\spheresp{k-1}{M}}}\left(
      \cochain{M}, \cochain{\spheresp{k-1}{M}}
    \right)
    \cong \cohom{\hom_{\msphere{k-1}}\left(
        \mdisk{k}, \msphere{k-1}
      \right)}
  \end{equation}
\end{proposition}

Moreover, we define
\(\msphere{k-1}(\bi) = \wedge V(\bi) \otimes \wedge \susp[k-1]{V(\bi)}\)
and
\(\mdisk{k}(\bi) = \wedge V(\bi) \otimes \wedge \susp[k-1]{V(\bi)} \otimes \susp[k]{V(\bi)}\).
Then we have
\(\mainterm\colon V(\bi)\to \msphere{k-1}(\bi)\)
and
\(\errterm\colon V(\bi)\to \mdisk{k}(\bi - 1)\).

\newcommand{\elmwedgeU}{\nu}
\newcommand{\fst}[1]{{#1}_{(1)}}
\newcommand{\snd}[1]{{#1}_{(2)}}

Next we give a construction of shriek maps.

\begin{definition}
  \label{definition:shriekInductiveConstruction}
  For \(\bi = 0,\ldots, n-1\)
  define a \(\K\)-linear map
  \begin{equation}
    \Phi\colon
    \hom_{\msphere{k-1}(\bi-1)}
    \left(
      \mdisk{k}(\bi-1),
      \msphere{k-1}(\bi-1)
    \right)
    \rightarrow \hom_{\msphere{k-1}(\bi)}
    \left(
      \mdisk{k}(\bi),
      \msphere{k-1}(\bi)
    \right)
  \end{equation}
  of odd degree as follows.
  \begin{enumerate}
    \item In the case \(\deg{\base_{\bi}}+k-1\) is odd,
      for
      \begin{math}
        f \in \hom_{\msphere{k-1}(\bi-1)}
        \left(
          \mdisk{k}(\bi-1),
          \msphere{k-1}(\bi-1)
        \right),
      \end{math}
      define
      \begin{equation}
        \Phi(f) \in \hom_{\msphere{k-1}(\bi)}
        \left(
          \mdisk{k}(\bi),
          \msphere{k-1}(\bi)
        \right)
      \end{equation}
      by
      \begin{equation}
        \Phi(f)(\elmwedgeU)
        = \mainterm \base_{\bi} \cdot f(\elmwedgeU)
        - (-1)^{\deg{f}}f(\errterm \base_{\bi} \cdot \elmwedgeU),\quad
        \Phi(f)(\elmwedgeU\cdot(\susp{\base_{\bi}})^l)=0
      \end{equation}
      for \(\elmwedgeU \in \wedge \susp{V(\bi-1)}\) and \(l \geq 1\).
    \item In the case \(\deg{\base_{\bi}}+k-1\) is even,
      for
      \begin{math}
        f \in \hom_{\msphere{k-1}(\bi-1)}
        \left(
          \mdisk{k}(\bi-1),
          \msphere{k-1}(\bi-1)
        \right),
      \end{math}
      define \(\Phi(f)\) by
      \begin{equation}
        \Phi(f)(\elmwedgeU\cdot\susp{\base_{\bi}})
        = (-1)^{\deg{f} + \deg{\elmwedgeU}}f(\elmwedgeU),\quad
        \Phi(f)(\elmwedgeU) = 0
      \end{equation}
      for \(\elmwedgeU \in \wedge \susp{V(\bi-1)}\).
  \end{enumerate}
\end{definition}

By a straight-forward calculation,
the linear map \(\Phi\) is a chain map of odd degree.
In other words, the map \(\Phi\) satisfies \(d\Phi = -\Phi d\).

\newcommand{\extgen}{\varphi}
\newcommand{\extgeni}[1]{\extgen_{#1}} 
Hence we define chain maps
\begin{equation}
  \extgeni{\bi} \in \hom_{\msphere{k-1}(\bi)}
        \left(
          \mdisk{k}(\bi),
          \msphere{k-1}(\bi)
        \right)
\end{equation}
by
\(\extgeni{0} = \id_\K\)
and
\(\extgeni{\bi+1} = \Phi(\extgeni{\bi})\),
inductively.


\subsection{The pure case with \(k\) odd}
\label{subsection:modelOfShriek_pure}

Next we investigate the above map
in the case \((\wedge V, d)\) is pure and \(k\) is odd.

\begin{definition}[{\cite[Section 32 (a)]{felix-halperin-thomas01}}]
  A Sullivan algebra \((\wedge V, d)\) is \textit{pure}
  if \(d(\evenpart{V}) = 0\) and \(d(\oddpart{V}) \subset \wedge \evenpart{V}\).
\end{definition}

Here we apply the above construction for the case
the basis \(\base_1,\ldots,\base_n\)
is given by the sequence
\(x_1,\ldots x_p,y_1,\ldots, y_q\),
where
\(x_1,\ldots x_p\) and \(y_1,\ldots, y_q\) are
(arbitrary) bases of \(\evenpart{V}\) and \(\oddpart{V}\), respectively.
That is,
\(\base_i = x_i\) for \(1\leq i \leq p\) and
\(\base_{p+j} = x_j\) for \(1\leq j \leq q\).
In this case,
we can write
\begin{math}
  \errterm y_j = (-1)^k\sum_i \alpha_{ji}\cdot \susp[k]{x_i}
\end{math}
for some elements \(\alpha_{ji} \in \msphere{k-1}\).
Note that \(\alpha_{ji} \in \wedge\evenpart{V}\) when \(k\geq 2\),
and \(\alpha_{ji} \in \wedge\evenpart{V}\otimes\wedge\evenpart{V}\) when \(k = 1\).

\newcommand{\inclconstmodel}{\mu}
Let \(\inclconstmodel\colon \msphere{k-1}\to\wedge V\) be
the multiplication map when \(k=1\),
and the map defined by
\(\inclconstmodel(v) = v\) and \(\inclconstmodel(\susp[k-1]{v}) = 0\)
when \(k\geq 2\).
Then we have
\begin{equation}
  \label{equation:partial_derivative}
  \inclconstmodel(\alpha_{ji}) = \frac{\partial(dy_j)}{\partial x_i} \in \wedge \evenpart{V}.
\end{equation}

\newcommand{\len}[1]{l(#1)}
Write \([p] = \{1, 2, \ldots, p\}\).
For any subset
\(I = \{i_1, \ldots, i_n\} \subset [p]\) with \(i_1 < \cdots < i_n\),
we define
\(\deg{I} = i_1 + \cdots i_n\),
\(\len{I} = n\), and
\(\susp[k]{x_I} = \susp[k]{x_{i_1}}\cdots\susp[k]{x_{i_n}}\).
Similarly, for any subset
\(I = \{j_1, \ldots, j_n\} \subset [q]\) with \(j_1 < \cdots < j_n\),
we define
\(\mainterm{y_J} = \mainterm{y_{j_n}}\cdots\mainterm{y_{j_1}}\).

For \(0\leq i \leq p\),
we can easily compute \(\extgeni{i}\) by induction on \(i\).

\begin{lemma}
  \label{lemma:extgeni_evenpart}
  For any integer \(i\) with \(0\leq i \leq p\)
  and any subset \(I\subset [i]\),
  we have
  \begin{equation}
    \extgeni{i}(\susp[k]{x_{[p]\setminus I}}) =
    \begin{cases}
      1, & \text{if \(I = \emptyset\),}\\
      0, & \text{if \(I \neq \emptyset\).}
    \end{cases}
  \end{equation}
\end{lemma}

Moreover,
we have the following formulas for \(\extgeni{p+j}\) for \(0\leq j \leq q\).

\begin{proposition}
  \label{proposition:extgeni_oddpart}
  Let \(j\) be an integer with \(0\leq j \leq q\) and
  \(I \subset [p]\) a subset.
  Write \(n = \len{I}\) and \(I = \{i_1, \ldots, i_n\}\)
  with \(i_1 < \cdots < i_n\).
  Then the element
  \(\extgeni{p+j}(\susp[k]{x_{[p]\setminus I}}) \in \msphere{k-1}(j)\)
  satisfies the following.
  \begin{enumerate}
    \item \label{item:extgeni_top}
      If \(n = 0\), then we have
      \(\extgeni{p+j}(\susp[k]{x_{[p]}}) = \mainterm y_{[j]}\).
    \item \label{item:extgeni_middle}
      If \(n < j\),
      then the element \(\extgeni{p+j}(\susp[k](x_{[p]\setminus I}))\)
      is contained in the ideal
      \((\mainterm y_1,\ldots, \mainterm y_j) \subset \msphere{k-1}(j)\).
    \item \label{item:extgeni_bottom}
      If \(n \geq j\), then we have
      \begin{equation}
        \extgeni{p+j}(\susp[k]{x_{[p]\setminus I}}) =
        \begin{cases}
          (-1)^{\deg{I}+pj}\det\bigl((\alpha_{t, i_r})_{1\leq t, r\leq j}\bigr),
          & \text{if \(n = j\),} \\
          0, & \text{if \(n > j\).}
        \end{cases}
      \end{equation}
  \end{enumerate}
\end{proposition}
\begin{proof}
  We prove the formulas by induction on \(j\).
  The case \(j=0\) is already proved in \cref{lemma:extgeni_evenpart}.
  Assume that \(j\geq 1\) and we already have the formulas for \(\extgeni{p+j-1}\).

  By \cref{definition:shriekInductiveConstruction},
  we have
  \begin{align}
    & \extgeni{p+j}(\susp[k]{x_{[p]\setminus I}})
    = \Phi(\extgeni{p+j-1})(\susp[k]{x_{[p]\setminus I}}) \\
    &= \mainterm y_j\cdot \extgeni{p+j-1}(\susp[k]{x_{[p]\setminus I}})
      + (-1)^{p+j-1} \extgeni{p+j-1}(\errterm y_j\cdot \susp[k]{x_{[p]\setminus I}}) \\
    &= \mainterm y_j\cdot \extgeni{p+j-1}(\susp[k]{x_{[p]\setminus I}})
      + (-1)^{p+j-1} \extgeni{p+j-1}\left(
      (-1)^k\sum_{1\leq i\leq p} \alpha_{ji}\cdot \susp[k]{x_i} \cdot \susp[k]{x_{[p]\setminus I}}
      \right) \\
    &= \mainterm y_j\cdot \extgeni{p+j-1}(\susp[k]{x_{[p]\setminus I}})
      + (-1)^{p+j}\sum_{1\leq r\leq j}(-1)^{i_r-r}\alpha_{j,i_r}\cdot\extgeni{p+j-1}
      \left(\susp[k]{x_{[p]\setminus I_r}}\right),
      \label{align:extgeni_induction}
  \end{align}
  where \(I_r = I\setminus \{i_r\}\).

  First we prove \cref{item:extgeni_top}.
  Since \(\deg{\susp[k]{x_i}}\) is odd, we have
  \begin{math}
    \errterm y_j\cdot \susp[k]{x_{[p]}}
    = (-1)^k\sum_i \alpha_{ji}\cdot \susp[k]{x_i} \cdot \susp[k]{x_{[p]}}
    = 0.
  \end{math}
  Hence we have
  \begin{align}
    \extgeni{p+j}(\susp[k]{x_{[p]}})
    &= \Phi(\extgeni{p+j-1})(\susp[k]{x_{[p]}}) \\
    &=\mainterm y_j\cdot \extgeni{p+j-1}(\susp[k]{x_{[p]}})
      \pm \extgeni{p+j-1}(\errterm y_j\cdot \susp[k]{x_{[p]}}) \\
    &= \mainterm y_j \cdot \mainterm y_{[j-1]} = \mainterm y_{[j]}.
  \end{align}

  Next we prove \cref{item:extgeni_middle}.
  Assume \(n < j\).
  Then, for any \(r\),
  we have
  \begin{math}
    \extgeni{p+j-1}\left(\susp[k]{x_{[p]\setminus I_r}}\right)
    \in (\mainterm y_1,\ldots, \mainterm y_{j-1})
  \end{math}
  by the induction hypothesis,
  since \(\len{I_r} = n-1 < j-1\).
  Thus we have
  \begin{math}
    \extgeni{p+j}(\susp[k](x_{[p]\setminus I}))
    \in (\mainterm y_1,\ldots, \mainterm y_j)
  \end{math}
  by \cref{align:extgeni_induction}.

  Finally we prove \cref{item:extgeni_bottom}.
  Assume \(n \geq j\).
  Since \(\len{I} = n > j-1\),
  we have \(\extgeni{p+j-1}(\susp[k]{x_{[p]\setminus I}}) = 0\)
  by the induction hypothesis.
  Hence \cref{align:extgeni_induction} reduces to the equation
  \begin{equation}
    \label{equation:extgeni_bottom}
    \extgeni{p+j}(\susp[k]{x_{[p]\setminus I}})
      = (-1)^{p+j}\sum_{1\leq r\leq j}(-1)^{i_r-r}\alpha_{j,i_r}\cdot\extgeni{p+j-1}
      \left(\susp[k]{x_{[p]\setminus I_r}}\right).
  \end{equation}


  If \(n > j\),
  since \(\len{I_r} = n-1 > j-1\),
  we have
  \begin{math}
    \extgeni{p+j-1}(\susp[k]x_{[p]\setminus I_r}) = 0
  \end{math}
  and hence
  \begin{math}
    \extgeni{p+j}(\susp[k]{x_{[p]\setminus I}}) = 0
  \end{math}
  by \cref{equation:extgeni_bottom}.
  This proves \cref{item:extgeni_bottom} in the case \(n > j\).

  \newcommand{\minor}{M}
  Next we assume \(n = j\).
  Let \(\minor_{u,r}\) be the minor determinants of
  the \(j\times j\) matrix
  \begin{math}
    A = \left(
      \alpha_{t,i_s}
    \right)_{1\leq t,s\leq j},
  \end{math}
  i.e.\ %
  \begin{math}
    \minor_{u,r} = \det\left(
      (\alpha_{t,i_s})_{t\neq u, s\neq r}
    \right).
  \end{math}
  Since \(\deg{I_r} = j-1\),
  we have
  \begin{math}
    \extgeni{p+j-1}(\susp[k]x_{[p]\setminus I_r})
    = (-1)^{\deg{I_r}+p(j-1)}\minor_{j,r}
  \end{math}
  by the induction hypothesis.
  Hence, by \cref{equation:extgeni_bottom}, we have
  \begin{align}
    \extgeni{p+j}(\susp[k]{x_{[p]\setminus I}})
    &= (-1)^{p+j}\sum_{1\leq r\leq j}(-1)^{i_r-r}\alpha_{j,i_r}\cdot(-1)^{\deg{I_r}+p(j-1)}\minor_{j,r} \\
    &= (-1)^{\deg{I}+pj}\sum_{1\leq r\leq j} (-1)^{j+r}\minor_{j,r} \\
    &= (-1)^{\deg{I}+pj} \det\bigl((\alpha_{t, i_r})_{1\leq t, r\leq j}\bigr).
  \end{align}
  This proves \cref{item:extgeni_bottom} in the case \(n=j\).
\end{proof}

\begin{proposition}
  \label{proposition:nonTriviality_pure}
  If
  \begin{math}
    \extgen \in
    \hom_{\msphere{k-1}}\left(\mdisk{k}, \msphere{k-1}\right)
  \end{math}
  is a chain map satisfying
  \begin{math}
    \extgen(\susp[k]{x_{[p]}}) = \mainterm y_{[q]},
  \end{math}
  then we have
  \begin{equation}
    [\extgen] \neq 0 \in \ext_{\msphere{k-1}}\left(\wedge V, \msphere{k-1}\right).
  \end{equation}
\end{proposition}
\begin{proof}
  Let \(I \subset \msphere{k-1}\) be the ideal generated by
  \begin{math}
    x_1\otimes 1, \ldots, x_p\otimes 1,
    y_1\otimes 1, \ldots, y_q\otimes 1,
    \mainterm x_1, \ldots, \mainterm x_p.
  \end{math}
  Note that \(d(I) \subset I\)
  since \((\wedge V, d)\) is pure.
  Consider the evaluation map
  \begin{equation}
    \ev\colon
    \ext_{\msphere{k-1}}(\wedge V, \msphere{k-1})
    \otimes \tor_{\msphere{k-1}}(\wedge V, \msphere{k-1} / I)
    \to \tor_{\msphere{k-1}}(\msphere{k-1}, \msphere{k-1} / I)
    \cong \wedge (\mainterm y_1, \ldots, \mainterm y_q).
  \end{equation}
  Using \(\mdisk{k}\) as a resolution of \((\wedge V, d)\) over \(\msphere{k-1}\),
  we have elements
  \([\extgen] \in \ext_{\msphere{k-1}}(\wedge V, \msphere{k-1})\)
  and
  \([\susp[k]{x_{[p]}}\otimes 1] \in \tor_{\msphere{k-1}}(\wedge V, \msphere{k-1} / I)\).
  Then we have
  \begin{equation}
    \ev([\extgen]\otimes [\susp[k]{x_{[p]}}\otimes 1])
    = \mainterm y_{[q]} \neq 0\in \wedge (\mainterm y_1, \ldots, \mainterm y_q).
  \end{equation}
  This proves the proposition.
\end{proof}

\begin{corollary}
  \label{corollary:shriekJacobian}
  Assume \(p \leq q\), i.e.\ \(\dim\evenpart{V} \leq \dim\oddpart{V}\).
  Then there is a chain map
  \begin{math}
    \extgen \in
    \hom_{\msphere{k-1}}\left(\mdisk{k}, \msphere{k-1}\right)
  \end{math}
  such that
  \begin{enumerate}
    \item
      \begin{math}
        [\extgen] \neq 0 \in
        \ext_{\msphere{k-1}}\left(\wedge V, \msphere{k-1}\right).
      \end{math}
    \item
      \begin{math}
        \inclconstmodel \circ \extgen(1) =
        \begin{cases}
          \det\left(
            \left(
              \frac{\partial(dy_j)}{\partial x_i}
            \right)_{1\leq i, j\leq p}
          \right)
          \in \wedge \evenpart{V}, & \text{if \(p=q\),} \\
          0, & \text{if \(p<q\).}
        \end{cases}
      \end{math}
  \end{enumerate}
\end{corollary}
\begin{proof}
  Define
  \begin{math}
    \extgen = (-1)^{\frac{1}{2}p(p+3)} \extgeni{2p}.
  \end{math}
  By \cref{proposition:extgeni_oddpart} \cref{item:extgeni_top}
  and \cref{proposition:nonTriviality_pure},
  we have
  \begin{math}
    [\extgen] \neq 0 \in
    \ext_{\msphere{k-1}}\left(\wedge V, \msphere{k-1}\right).
  \end{math}
  If \(p = q\),
  by \cref{equation:partial_derivative}
  and \cref{proposition:extgeni_oddpart} \cref{item:extgeni_bottom},
  we have
  \begin{math}
    \inclconstmodel \circ \extgen(1)
    = \det\left(
      \frac{\partial(dy_j)}{\partial x_i}
    \right).
  \end{math}
  If \(p < q\),
  by \cref{proposition:extgeni_oddpart} \cref{item:extgeni_middle},
  we have
  \begin{math}
    \inclconstmodel \circ \extgen(1) = 0
  \end{math}
  since \(\mainterm y_j \in \ker \inclconstmodel\).
\end{proof}

\begin{remark}
  We can generalize the non-triviality of the chain map \(\extgen = \extgeni{\dim V}\)
  using the method and notion given in \cite{wakatsuki16}.
  Let \((\wedge V, d)\) be a \textit{semi-pure} Sullivan algebra,
  i.e.\ \(\dim V < \infty\) and
  \(d(\evenpart{V})\) is contained in
  the ideal \(\wedge V\cdot \evenpart{V}\) generated by \(\evenpart{V}\).
  Take bases \(x_1,\ldots,x_p\) and \(y_1,\ldots,y_q\)
  of \(\evenpart{V}\) and \(\oddpart{V}\), respectively.
  By induction on \(\dim V\), we have
  \(\extgen(\susp[k]{x_{[p]}}) = \mainterm y_{[q]}\)
  along with
  \(\extgen(\elmwedgeU) = 0\) for any
  \begin{math}
    \elmwedgeU\in (\susp[k]{y_1},\ldots,\susp[k]{y_q})
    \subset \mdisk{k}.
  \end{math}
  The first equation
  \([\extgen] \neq 0 \in \ext_{\msphere{k-1}}(\wedge V, \msphere{k-1})\),
  since \cref{proposition:nonTriviality_pure} also holds for a semi-pure Sullivan algebra.
\end{remark}

\subsection{Proof of \cref{proposition:shriekInclconst_EulerChar}}
In this subsection, we prove \cref{proposition:shriekInclconst_EulerChar}
using the chain map in \cref{corollary:shriekJacobian}.

\begin{definition}[{\cite[Section 32]{felix-halperin-thomas01}}]
  A 1-connected space \(M\) is \textit{rationally elliptic} if
  \(\dim_\K\left(\bigoplus_n\cohom[n]{M}\otimes\K\right) < \infty\) and
  \(\dim_\K\left(\bigoplus_n\pi_n(M)\otimes\K\right) < \infty\).
\end{definition}

First we recall a fundamental theorem on rationally elliptic space.
\begin{theorem}[{\cite[Proposition 32.16]{felix-halperin-thomas01}}]
  \label{theorem:elliptic}
  Let \(M\) be a rationally elliptic space.
  Then we have
  \begin{itemize}
    \item \(\chi(M) \geq 0\) and
    \item
      \begin{math}
        \dim_\K\left(\bigoplus_n\pi_{2n}(M)\otimes\K\right)
        \leq \dim_\K\left(\bigoplus_n\pi_{2n-1}(M)\otimes\K\right).
      \end{math}
  \end{itemize}
  Moreover, the following conditions are equivalent:
  \begin{enumerate}
    \item \(\chi(M) > 0\).
    \item
      \begin{math}
        \dim_\K\left(\bigoplus_n\pi_{2n}(M)\otimes\K\right)
        < \dim_\K\left(\bigoplus_n\pi_{2n-1}(M)\otimes\K\right).
      \end{math}
    \item \label{item:pureModel}
      The minimal Sullivan model \((\wedge V, d)\) of \(M\) is pure,
      \(\dim\evenpart{V}=\dim\oddpart{V}=p\), and
      \(dy_1,\ldots,dy_p\) is a regular sequence in \(\wedge\evenpart{V}\),
      where \(\oddpart{V} = \spanv_\K\{y_1,\ldots,y_p\}\).
  \end{enumerate}
\end{theorem}

Using the theorem with the construction given in \cref{subsection:modelOfShriek_pure},
we have the following proposition.

\begin{proposition}
  \label{proposition:NonZeroJacobian}
  Let \(M\) be a rationally elliptic space satisfying the conditions in \cref{theorem:elliptic},
  and \((\wedge V, d)\) its minimal Sullivan model.
  Write
  \(\evenpart{V} = \spanv_\K\{x_1,\ldots,x_p\}\) and
  \(\oddpart{V} = \spanv_\K\{y_1,\ldots,y_p\}\).
  Then we have
  \begin{equation}
    \left[
      \det\left(
        \left(
          \frac{\partial(dy_j)}{\partial x_i}
        \right)_{1\leq i, j\leq p}
      \right)
    \right]
    \neq 0 \in \cohom{\wedge V} \ \bigl(\cong \cohom{M}\bigr).
  \end{equation}
\end{proposition}
\begin{proof}
  By \cref{corollary:shriekJacobian} for \(k=1\),
  we have a chain map
  \begin{math}
    \extgen \in \hom_{\wedge V\tpow2}(\wedge V, \wedge V\tpow2)
  \end{math}
  such that
  \begin{math}
    [\extgen] \neq 0 \in
    \ext^m_{\wedge V\tpow2}(\wedge V, \wedge V\tpow2)
  \end{math}
  and
  \begin{math}
    \inclconstmodel\circ\extgen(1) = \det
    \left(
      \frac{\partial(dy_j)}{\partial x_i}
    \right)
    \in \wedge \evenpart{V}.
  \end{math}
  Since \(\inclconstmodel\) is a model of \(\diag\colon M\to M\times M\),
  we have
  \begin{equation}
    \diag^*\circ(\exttohom{[\extgen]})(1) =
    [\inclconstmodel\circ\extgen(1)] =
    \left[\det
      \left(
        \frac{\partial(dy_j)}{\partial x_i}
      \right)
    \right].
  \end{equation}
  Since
  \begin{math}
    \ext^m_{\wedge V\tpow2}(\wedge V, \wedge V\tpow2)
    \cong \ext^m_{\cochain{M\times M}}(\cochain{M},\cochain{M\times M})
    \cong \K,
  \end{math}
  we have
  \begin{math}
    [\extgen] = \shriek\diag
  \end{math}
  (up to scalar multiplication).
  Hence by \cref{proposition:diagonalClass}, we have
  \begin{equation}
    \left[\det
      \left(
        \frac{\partial(dy_j)}{\partial x_i}
      \right)
    \right]
    = \chi(M)\omega.
  \end{equation}
  Since \(\chi(M)\neq 0\) by \cref{item:pureModel} of \cref{theorem:elliptic},
  this proves the proposition.
\end{proof}

\begin{remark}
  The proposition also follows from \cite[Proposition 3]{smith82}.
  Here we give an alternative proof using an idea coming from string topology.
\end{remark}

Now we give a proof of \cref{proposition:shriekInclconst_EulerChar},
which completes the proof of \cref{theorem:vanishingMainTheorem}.

\begin{proof}[Proof of \cref{proposition:shriekInclconst_EulerChar}]
  Since the statement is trivial when \(\chi(M)=0\),
  we may assume \(\chi(M)\neq 0\).
  Then, by \cref{theorem:elliptic},
  the minimal Sullivan model \((\wedge V, d)\) of \(M\)
  satisfies \cref{item:pureModel}.
  Take
  \begin{math}
    \extgen \in
    \hom_{\msphere{k-1}}\left(\mdisk{k}, \msphere{k-1}\right)
  \end{math}
  by \cref{corollary:shriekJacobian}.
  Then we have
  \begin{math}
    \inclconst^*\circ(\exttohom{[\extgen]})(1) \neq 0
    \in \cohom{\wedge V} \cong \cohom{M}
  \end{math}
  by \cref{proposition:NonZeroJacobian}.
  Thus \(\myshriek=[\extgen]\) satisfies the equation
  (after multiplication of a non-zero scalar, if necessary).
\end{proof}

\section*{Acknowledgment}
This work is based on discussions with Alexander Berglund at Stockholm University,
which was financially supported by the Program for Leading Graduate School, MEXT, Japan.
I would like to express my gratitude to Katsuhiko Kuribayashi and Takahito Naito for productive discussions and valuable suggestions.
Furthermore, I would like to thank my supervisor Nariya Kawazumi
for the enormous encouragement and comments.
This work was supported by JSPS KAKENHI Grant Number 16J06349.

\makeatletter
\renewenvironment{thebibliography}[1]
{\section*{\refname\@mkboth{\refname}{\refname}}%
  \list{\@biblabel{\@arabic\c@enumiv}}%
       {\settowidth\labelwidth{\@biblabel{#1}}%
        \leftmargin\labelwidth
        \advance\leftmargin\labelsep
        \setlength\itemsep{0pt}
        \setlength\baselineskip{11pt}
        \@openbib@code
        \usecounter{enumiv}%
        \let\p@enumiv\@empty
        \renewcommand\theenumiv{\@arabic\c@enumiv}}%
  \sloppy
  \clubpenalty4000
  \@clubpenalty\clubpenalty
  \widowpenalty4000%
  \sfcode`\.\@m}
 {\def\@noitemerr
   {\@latex@warning{Empty `thebibliography' environment}}%
  \endlist}
\makeatother

\bibliographystyle{halphaabbrv}
\bibliography{references}

\begin{thebibliography}{FHT01}

\bibitem[Ber15]{berglund15}
Alexander Berglund.
\newblock Rational homotopy theory of mapping spaces via {L}ie theory for
  {$L_\infty$}-algebras.
\newblock {\em Homology Homotopy Appl.}, 17(2):343--369, 2015.

\bibitem[CG04]{cohen-godin}
Ralph~L. Cohen and V{\'e}ronique Godin.
\newblock A polarized view of string topology.
\newblock In {\em Topology, geometry and quantum field theory}, volume 308 of
  {\em London Math. Soc. Lecture Note Ser.}, pages 127--154. Cambridge Univ.
  Press, Cambridge, 2004.

\bibitem[CS99]{chas-sullivan}
Moira Chas and Dennis Sullivan.
\newblock String topology, 1999, arXiv:math/9911159.

\bibitem[FHT88]{felix-halperin-thomas88}
Yves F{\'e}lix, Stephen Halperin, and Jean-Claude Thomas.
\newblock Gorenstein spaces.
\newblock {\em Adv. in Math.}, 71(1):92--112, 1988.

\bibitem[FHT01]{felix-halperin-thomas01}
Yves F{\'e}lix, Stephen Halperin, and Jean-Claude Thomas.
\newblock {\em Rational homotopy theory}, volume 205 of {\em Graduate Texts in
  Mathematics}.
\newblock Springer-Verlag, New York, 2001.

\bibitem[FT09]{felix-thomas09}
Yves F{\'e}lix and Jean-Claude Thomas.
\newblock String topology on {G}orenstein spaces.
\newblock {\em Math. Ann.}, 345(2):417--452, 2009.

\bibitem[Men13]{menichi13}
Luc Menichi.
\newblock String topology, {E}uler class and {TNCZ} free loop fibrations, 2013,
  arXiv:1308.6684.

\bibitem[MS74]{milnor-stasheff}
John~W. Milnor and James~D. Stasheff.
\newblock {\em Characteristic classes}.
\newblock Princeton University Press, Princeton, N. J.; University of Tokyo
  Press, Tokyo, 1974.
\newblock Annals of Mathematics Studies, No. 76.

\bibitem[Nai13]{naito13}
Takahito Naito.
\newblock String operations on rational {G}orenstein spaces, 2013,
  arXiv:1301.1785.

\bibitem[Smi67]{smith67}
Larry Smith.
\newblock Homological algebra and the {E}ilenberg-{M}oore spectral sequence.
\newblock {\em Trans. Amer. Math. Soc.}, 129:58--93, 1967.

\bibitem[Smi82]{smith82}
Larry Smith.
\newblock A note on the realization of graded complete intersection algebras by
  the cohomology of a space.
\newblock {\em Quart. J. Math. Oxford Ser. (2)}, 33(131):379--384, 1982.

\bibitem[Tam10]{tamanoi}
Hirotaka Tamanoi.
\newblock Loop coproducts in string topology and triviality of higher genus
  {TQFT} operations.
\newblock {\em J. Pure Appl. Algebra}, 214(5):605--615, 2010.

\bibitem[Wak]{wakatsuki18:toappear}
Shun Wakatsuki.
\newblock Coproducts in brane topology.
\newblock to appear in Algebr. Geom. Topol., also available at
  arXiv:1802.04973.

\bibitem[Wak16]{wakatsuki16}
Shun Wakatsuki.
\newblock Description and triviality of the loop products and coproducts for
  rational {G}orenstein spaces, 2016, arXiv:1612.03563.

\end{thebibliography}

\end{document}